\documentclass[a4paper]{amsart}

\usepackage[backend=biber,maxbibnames=5,maxalphanames=5,style=alphabetic,bibencoding=utf8,giveninits,url=true,isbn=false,doi=true]{biblatex}
\renewbibmacro{in:}{}
\AtBeginBibliography{\small}

\usepackage{amsmath}
\usepackage{amssymb}
\usepackage{placeins}
\usepackage{mathtools}
\usepackage{enumitem}
\usepackage[font=footnotesize,margin=0.2cm]{caption}
\usepackage{subcaption}
\usepackage{amsfonts,amsthm}
\usepackage[foot]{amsaddr}
\usepackage[thinlines]{easytable}

\usepackage[colorlinks=true]{hyperref}
\hypersetup{citecolor=darkcerulean,linkcolor=darkcerulean,    urlcolor=darkcerulean}

\usepackage{tikz}
\usepackage{pgfplots}
\pgfplotsset{compat=1.16}
\usetikzlibrary{angles,quotes}
\usetikzlibrary{intersections, pgfplots.fillbetween}
\usetikzlibrary{backgrounds}
\usetikzlibrary{matrix}

\usepackage{graphicx}
\definecolor{cobalt}{rgb}{0.0, 0.28, 0.67}
\definecolor{darkcerulean}{rgb}{0.03, 0.27, 0.49}

\newcommand{\omitit}[1]{}
\newcommand{\ebar}{\overline{e}}
\newcommand{\tbar}{\overline{t}}
\newcommand{\tet}{\tau}

\newcommand{\mesh}{{\mathcal T}}
\newcommand{\meshW}{\mesh_{{W}}}
\newcommand{\meshA}{\mesh_{{A}}}
\newcommand{\meshP}{\mesh_{{P}}}

\newcommand{\curl}{{{\rm curl} \,}}
\newcommand{\sdiv}{{{\nabla\cdot} \,}}
\newcommand{\vv}{{\bf v}}
\newcommand{\set}[2]{\left\lbrace #1 \; : \; #2 \right\rbrace}
\newcommand{\half}{\tfrac{1}{2}}

\newcommand{\vhk}{{\bf V}}
\newcommand{\bS}{{\bf S}}
\newcommand{\bUps}{{\bf \Upsilon}}
\newcommand{\bZ}{{\bf Z}}
\newcommand{\piq}{\Pi}
\newcommand{\edgnam}{\epsilon}

\newtheorem{Theorem}{Theorem}%[section]
\newtheorem{theorem}[Theorem]{Theorem}
\newtheorem{lemma}[Theorem]{Lemma}

\newtheorem{proposition}[Theorem]{Proposition}

\newtheorem{remark}[Theorem]{Remark}
\newtheorem{example}[Theorem]{Example}

\makeatletter
\@namedef{subjclassname@2020}{%
  \textup{2020} Mathematics Subject Classification}
\makeatother

\begin{document}

\title{Dimensions of exactly divergence-free finite element spaces in 3D}
\date{\today}

\author[L.~R.~Scott]{L.~Ridgway Scott}
\author[T.~Tscherpel]{Tabea Tscherpel}

\address[L.~R.~Scott]{The University of Chicago, Emeritus, Chicago, Illinois, 60637}
\address[T.~Tscherpel]{Department of Mathematics, Technische Universität Darmstadt, Dolivostr. 15, 64293 Darmstadt,	Germany}

\email{ridg@uchicago.edu}
\email{tscherpel@mathematik.tu-darmstadt.de}

\begin{abstract}
We examine the dimensions of various inf-sup stable mixed finite element spaces on tetrahedral meshes in 3D with exact divergence constraints. 
More precisely, we compare the standard Scott--Vogelius elements of higher polynomial degree and low order methods on split meshes, the Alfeld and the Worsey--Farin split. 
	The main tool is a counting strategy to express the degrees of freedom for given polynomial degree and given split in terms of few mesh quantities, for which bounds and asymptotic behavior under mesh refinement is investigated. 
	Furthermore, this is used to obtain insights on potential precursor spaces in the Stokes
 complex for finite element methods on the Worsey--Farin split. 
\end{abstract}


\maketitle

%-------------------------------------------------------------------------------
\section{Introduction}

Mixed finite element spaces for the Stokes and Navier--Stokes equations have been developed 
and investigated since the `70s, see, e.g.,~\cite{TH.1973,CR.1973,ABF.1984, BR.1985,lrsBIBbk} 
and~\cite{BBF.2013} for a review. 
Inf-sup stability~\cite{lrsBIBgd} of the mixed pair of finite element spaces is a necessary condition to ensure well-posedness and stability of the discrete solutions. 
However, in general inf-sup stable pairs do not lead to exactly divergence-free discrete velocity functions. 
This is the case for most classical mixed pairs such as the Taylor--Hood~\cite{TH.1973}, the Crouzeix--Raviart~\cite{CR.1973}, the MINI~\cite{ABF.1984},
and the Bernardi--Raugel~\cite{BR.1985} elements. 
Yet, exact incompressibility constraints are beneficial for example for pressure robust approximation~\cite{john2016divergence}. 
The class of mixed finite element spaces with exact divergence constraints has been reviewed by Neilan in~\cite{neilan2020stokes}. 
As noted therein, in three dimensions only a few inf-sup stable pairs of spaces with exact divergence-constraints are available, and the analysis of such methods is still far from complete. 

In the following we consider an open bounded domain $\Omega \subset \mathbb{R}^3$ with polytopal boundary $\partial \Omega$.  
Throughout, $\mesh$ denotes a conforming tetrahedral triangulation (or mesh) of $\overline{\Omega} \subset \mathbb{R}^3$ into closed tetrahedra,  cf.~\cite[Ch.~2, ($\mesh_h$1)-($\mesh_h$5)]{Ci.2002}. 
See also~\cite[Sec.~2.3]{ref:Freudenthalulation} which refers to this as a 
{\em consistent} triangulation. 
This is then a pure simplicial $3$-complex in the language of algebraic topology. 

We consider pairs of finite element spaces consisting of 
\begin{itemize}
\item
a velocity space $\vhk_k(\mesh)\subset C(\overline{\Omega})$, which is the space of continuous vector-valued functions that are piecewise
polynomials of degree at most $k$ on the mesh $\mesh$, and
\item
a pressure space $\sdiv\vhk_k(\mesh) \subset L^\infty(\Omega)$ which is a subspace of $\piq_{k-1}(\mesh)$, the 
space of discontinuous piecewise polynomials of degree at most $k-1$ on the same tetrahedral mesh,
\end{itemize}
for $k \geq 1$. 
Since the pressure space is the divergence of the velocity space, discretely divergence-free velocity functions are indeed exactly divergence-free. 
Also, for both the inf-sup condition to be satisfied and the discretely divergence-free velocity functions to be exactly divergence-free, the pressure space necessarily has to be the space of divergence of the velocity functions, cf.~\cite{neilan2020stokes}. 
Starting from this ansatz the challenge is to characterize the pressure space.

The 2D version of these elements are the Scott--Vogelius elements~\cite{lrsBIBbk,ScottVogeliusA,GS.2019}. 
For those elements on general meshes $\mesh$ the pressure space can be locally characterized by considering the so-called singular vertices. 
Those are vertices that have (only) adjacent edges that lie on two straight lines. 
Furthermore, they are known to be inf-sup stable for polynomial degrees~$k\geq 4$, see~\cite{GS.2019}.
In 3D the situation is more challenging: the pressure spaces have not been characterized and inf-sup stability is available only in special cases. 
More specifically, it has been proved in~\cite{zhang2011divergence} that in 3D on a particular family of regular meshes  $\mesh_h$, the so-called Freudenthal triangulations~\cite{ref:Freudenthalulation}, the pair $(\vhk_k(\mesh_h), \sdiv \vhk_k(\mesh_h))$ satisfies the inf-sup condition for polynomial degrees $k\geq 6$. 
Recent computational experiments on this family of regular meshes~\cite{lrsBIBiu} suggest that the inf-sup condition holds for $k\geq 4$ with constant independently of the mesh size~$h$.  

For the purpose of minimizing the dimension of the discrete problem there has been significant interest in spaces of lower polynomial degree. 
To achieve inf-sup stability for standard piecewise polynomial functions for lower 
polynomial-degree $k$,
 split meshes have been considered, cf.~\cite{ref:zhang3DalfeldSplit,ref:quadraticPowellSabinTets,FGNZ.2022} and the review in~\cite{neilan2020stokes}. 
Even though the lower-order methods on split meshes are essentially the Scott--Vogelius method on a special mesh, for clarity we shall refer to them by the name of the split.  
We use the term Scott--Vogelius elements only to refer to the methods on the original mesh. 
The term `macro-element' can be used to describe the split elements, but this term is
also used in theoretical arguments where the grouping of sub-elements is more {\em ad hoc}.
Thus, we use the terminology `split method' to avoid confusion. 

Here we show that in 3D the discrete spaces for all but one  low-order method on split meshes considered to date have a substantially larger dimension than the 3D version of the Scott--Vogelius method using $(\vhk_k(\mesh),\sdiv\vhk_k(\mesh))$ for $k=4$ on the original mesh. 
Even for $k=6$, this pair has only at most 50\% higher dimension than most of the lowest-order split methods considered here. 
The only exception to this is the lowest-order case on the so-called Worsey--Farin split mesh. 
By comparing the actual number of degrees of freedom rather than the polynomial degree, we find that the Scott--Vogelius method is competitive or at least not much more costly than standard low-order methods. 
Also this argument has not even taken into account the improved approximability results available for higher order finite element spaces. 
It is notable that the only method that is more efficient is the lowest-order method. 
As soon as one considers methods of higher order than linear, the Scott--Vogelius element for $k = 4$ is the most efficient (by a factor $2$). 

Our observations suggest that there is need for further investigation of the Scott--Vogelius element in three dimensions. 
In particular insights into meshes with singularities are of interest, i.e., meshes for which $\sdiv\vhk_k(\mesh)$ is a strict subspace of $\piq_{k-1}(\mesh)$.  
Note that nearly singular meshes, that is meshes that are very close to singular meshes, typically have a large inf-sup constant leading to degraded performance~\cite{lrsBIBgd,lrsBIBih}. 

We focus our attention to methods on tetrahedral meshes that use a Lagrange space on the original or on some split mesh as velocity space 
and the divergence thereof as pressure space. 
Note however, that outside of this class there are alternative elements satisfying exact 
divergence-constraints, see,  e.g.,~\cite{falk2013stokes,guzman2014conforming,GN.2014b,neilan2016stokes}. 
Hence they are suited for pressure robust approximation, cf.~\cite{john2016divergence}. 
For a comparison of the dimensions we shall briefly report on low order exactly divergence-free and on only approximately divergence-free methods. 	
Nodal variables for hexahedral meshes~\cite{neilan2016stokes}
are easier to count thanks to the type of underlying subdivision. 
Since such meshes are typically structured, determining the number of degrees of freedom of finite element spaces thereon is relatively straight-forward.  

Our main tool for determining the number of degrees of freedom of the respective finite element spaces is a counting argument for mesh quantities in 3D. 
Similar methods have been applied in~\cite{neilan2020stokes} to show surjectivity of a discrete divergence operator and hence to prove discrete inf-sup stability. 
Furthermore, the author of~\cite{neilan2020stokes} has used such arguments to obtain insights into discrete Stokes complexes.  
The numbers of tetrahedra, faces, edges, and vertices in a mesh can be expressed in just two parameters, for example the number of vertices and the average number of edges emanating from each vertex in the mesh. 
On the latter we obtain estimates. 
Furthermore, under uniform mesh refinement their asymptotic behavior is investigated in~\cite{PR.2003.c,PR.2005}, which we shall review. 
Indeed, for a range of uniform mesh refinements
 the average mesh quantities converge to the ones attained by certain regular meshes, as the mesh is refined. 
 This includes the 3D red refinement~\cite{B.1995,Z.1995,moore1995adaptive}, bisection refinement such as the generalization of the newest vertex bisection~\cite{B.1991, K.1994} as well as the longest edge bisection~\cite{R.1991}, and the Carey--Plaza algorithm~\cite{PC.2000}. 
 Consequently, 
 we obtain estimates for the mesh quantities and then also on the number of degrees of freedom in the finite element spaces under consideration. 
 This allows for a comparison with other mixed finite element spaces. 
 
Counting degrees of freedom for finite element spaces only provides part of the information required to evaluate the effectiveness of the corresponding simulations.
The relationship between the size of the resulting linear systems and their computational cost is not simple. 
Depending on the implementation, the relevant dimension may be the total dimension of both velocity and pressure space (if the full saddle point problem is solved), or only the dimension of the velocity space (when the Iterated Penalty Method (IPM)~\cite{lrsBIBgd,lrsBIBih} is used). 
Furthermore, techniques such as static condensation~\cite{cockburn2016static,wilson1974static}
can reduce the computational cost.  
Another aspect of computational cost is related to the approximation degree.
Higher-order elements may achieve the same accuracy on a coarser mesh in some cases. 
Investigating these aspects is beyond the scope of this work and left to future work. 
Here we only determine the dimension of the discrete spaces.  

 The corresponding results in 2D are much simpler and will be provided for comparison. 

\subsubsection*{Outline} 
In section~\ref{sec:splitmesh} we review mesh splits available in the literature both in 2D and in 3D. 
For general 3D tetrahedral conforming meshes we present an approach to count the mesh quantities in section~\ref{sec:meshquantities}.  
The resulting formul{\ae} involve the average number $\ebar$ of edges adjacent to a vertex in the mesh. 
We derive bounds and review its behavior under certain mesh refinements to determine realistic values of $\ebar$. 
Applying the counting methods allows us to determine the dimensions of various mixed finite element pairs in section~\ref{sec:pposm}. 
Furthermore, applying the counting strategy yields insights into the Stokes complex for the Worsey--Farin split presented in section~\ref{sec:compldig}. 
Finally, section~\ref{sec:conclu} summarizes our conclusions. 

\section{Split meshes}
\label{sec:splitmesh}

There are a number of split meshes that allow for exact divergence constraints 
to be satisfied for low-order finite elements on them. 
Perhaps the earliest such method uses the two-dimensional Malkus crossed-triangle mesh
\cite{ref:MalkusHughesmixedunify,malkus1984linear}, which is based on subdividing quadrilaterals into four triangles and uses piecewise linear velocity functions. 
The Malkus crossed-triangle split introduces singular vertices~\cite{lrsBIBaf,lrsBIBbk,ScottVogeliusA} at the
cross-points in the center of each quadrilateral. 
Since then, a range of split-mesh methods have been developed for simplicial triangulations with the goal of achieving inf-sup stability for low polynomial degrees, see~\cite{neilan2020stokes} for a review. 
Let us summarize the split methods leading to exact divergence constraints which we investigate in the following. 
To motivate and clarify the discussion we shall begin with the two-dimensional situation and then proceed with the three-dimensional situation. 

\subsection{Splits in 2D}
There are two ways of splitting a triangle $\tet$. 

The \emph{Alfeld split} (sometimes referred to as \emph{barycentric split}) splits each triangle into three triangles and creates one new vertex. 
This is done by choosing one interior point, e.g., the barycenter of $\tet$, and introducing three edges that connect each vertex of $\tet$ with this interior point, see~Fig.~\ref{fig:alfeldsplit-2d} (left). 
The Alfeld split eliminates any singular or nearly singular vertices, cf.~\cite{arnold1992quadratic}. 

The \emph{Powell--Sabin split}~\cite{PS.1977,ref:GuzmanLisNeilanPowelSabin} splits each triangle $\tet$ into six triangles and is a refinement of the Alfeld split, cf.~Fig.~\ref{fig:alfeldsplit-2d} (right). 
Starting from an Alfeld split each triangle is split again by connecting the interior points of any two triangles in the original mesh that share an edge. 
 This results in a total of six new edges in the interior of $\tet$.  
It creates collinear edges in any two edge-adjacent triangles, and thus  introduces singular vertices at the intersection of the original edges with the newly introduced edges. 
One advantage of the Powell--Sabin split is that the angles of the resulting triangulation are more 
favorable than the ones in the Alfeld split, since large angles are bisected. 
Nevertheless, the Powell--Sabin split reduces the angles in the original triangulation. 
The following lemma is established by enumerating the mesh entities introduced by the split.

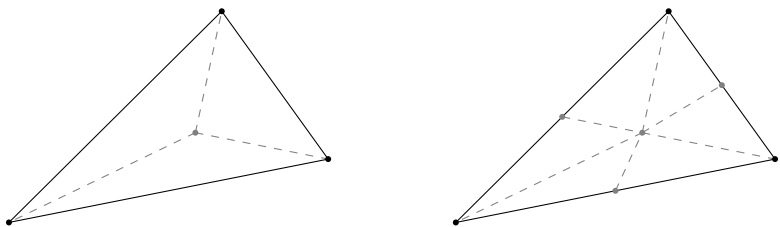
\begin{figure}[ht!]
	\centering
	\begin{tikzpicture}[scale=2.8]
		
		\begin{scope}
			\coordinate (T0) at (0,0);
			\coordinate (T2) at (1.5,0.3);
			\coordinate (T3) at (1,1);
			\coordinate (M) at (0.876,0.425);;
			\fill[gray] (M) circle (0.4pt);
			\draw[gray,dashed] (M) -- (T0);
			\draw[gray,dashed] (M) -- (T2);
			\draw[gray,dashed] (M) -- (T3);
			
			\draw (T0) -- (T3)--(T2) -- cycle;
			\fill (T0) circle (0.4pt);
			\fill (T2) circle (0.4pt);
			\fill (T3) circle (0.4pt);
		\end{scope}
		
		\begin{scope}[xshift = 2.1cm]
			
			\coordinate (T0) at (0,0);
			\coordinate (T2) at (1.5,0.3);
			\coordinate (T3) at (1,1);
			\coordinate (M) at (0.876,0.425);;
			\fill[gray] (M) circle (0.4pt);
			\draw[gray,dashed] (M) -- (T0);
			\draw[gray,dashed] (M) -- (T2);
			\draw[gray,dashed] (M) -- (T3);

			\draw (T0) -- (T3)--(T2) -- cycle;
			\fill (T0) circle (0.4pt);
			\fill (T2) circle (0.4pt);
			\fill (T3) circle (0.4pt);

			\coordinate (e0) at (1.25,0.65);
			\coordinate (e2) at (0.5,0.5);
			\coordinate (e3) at (0.75,0.15);
			% red
			\draw[gray,dashed] (M) -- (e0);
			\draw[gray,dashed] (e2) -- (M);
			\draw[gray,dashed] (e3) -- (M);
			
			\fill[gray] (e0) circle (0.4pt);
			\fill[gray] (e3) circle (0.4pt);	
			\fill[gray] (e2) circle (0.4pt);
						
		\end{scope}
	\end{tikzpicture}
	\caption{Alfeld split into $3$ new triangles (left) and  Powell--Sabin split split into $6$ new triangles each (right) in 2D with the new edges (dashed) and new vertices in gray. 
	}
	\label{fig:alfeldsplit-2d}
\end{figure}

\begin{lemma}[mesh quantities for 2D split meshes]\label{lem:meshquand-2D}
Let $\mesh$ be a conforming simplicial triangulation of a polyhedral bounded domain $\Omega \subset \mathbb{R}^2$ with $V$ vertices, $E$ edges, and $T$ triangles. 
Let $\meshA$ be the Alfeld split and $\meshP$ the Powell--Sabin split of $\mesh$. 
Then we have the following:
\begin{enumerate}[label=(\roman*)]
\item $\meshA$ has $V_A = V+T$ vertices, $E_A = E + 3T$ edges, and $T_A = 3 T$ triangles. 
\item $\meshP$ has $V_P = V+E + T$ vertices, $E_P = 2E + 6T$ edges, and $T_P = 6T$ triangles.  
\end{enumerate}
\end{lemma}

\subsection{Splits in 3D}
\label{sec:splithreed}

For a tetrahedron $\tet$ there are even more variants of barycentric subdivisions, some of which are the following. 

The \emph{Alfeld split}~\cite{alfeld2005c2,ref:zhang3DalfeldSplit,ref:GuzmaNeilanBarycntr} splits each tetrahedron in the mesh into $4$ tetrahedra. 
This is achieved by introducing $4$ new edges in any tetrahedron $\tet$, shown as gray dashed lines in Figure~\ref{fig:alfeldsplit} (left), that connect each vertex of $\tet$ to an interior
point, e.g., its barycenter. 
This creates one new vertex per tetrahedron. 
One defect of the Alfeld split is that it does not commute with a typical multigrid subdivision,
cf.~\cite{farrell2021reynolds}.
However, an optimal-order convergent, non-nested
multigrid method has been developed therein.

The \emph{Worsey--Farin split}~\cite{ref:cubicWorseyFarin} (also sometimes referred to as Clough--Tocher or Powell--Sabin~\cite{ref:quadraticPowellSabinTets,neilan2020stokes}) is a refinement of the Alfeld split and subdivides each of the $4$ subtetrahedra in the Alfeld split into $3$ subtetrahedra. 
Note that this is not the same as two successive Alfeld splits, which would result in $16$ subtetrahedra. 
Figure~\ref{fig:alfeldsplit} (right) depicts this further split, where the subtetrahedra of the Alfeld split are detached for the purpose of visualization. 
Each face of the original mesh is subdivided via a 2D Alfeld split resulting in one new vertex per face and $3$ new edges per face. 
The splitting point of the 2D Alfeld split of a face can be chosen as the point on the connecting line of the interior points of the two face neighboring tetrahedra, cf.~\cite{FGNZ.2022}. 
Then the face splitting point is connected to the interior point of the original simplex, resulting in $4$ new edges per tetrahedron in the Alfeld split. 
In this way, the original tetrahedron is divided into $12$ subtetrahedra.

Applying these splits to each tetrahedron in the triangulation $\mesh$ leads to new triangulation, denoted by $\meshA$ for the Alfeld split
and by $\meshW$ for the Worsey--Farin split. 
This creates on each face of the original mesh three singular edges, that is edges whose adjacent faces are lying (only) in two planes~\cite[Def.~4.1]{FGNZ.2022}. 
Singular edges in 3D are the direct generalization of singular vertices in 2D.

Alternatively, we may consider a barycentric subdivision of each $\tet$ into $24$ subtetrahedra.
Since this subdivision is a natural generalization of the Powell--Sabin split to three dimensions, it has been referred to as a generalized Powell--Sabin split 
\cite{ref:worseypipernotfarin}.

\begin{figure}[ht!]
\centering
\begin{tikzpicture}[scale=2.8]
	
	\begin{scope}
		\coordinate (T0) at (0,0);
		\coordinate (T1) at (1,0);
		\coordinate (T2) at (1.5,0.3);
		\coordinate (T3) at (1,1);
		\coordinate (M) at (0.876,0.325);
		\coordinate (F3) at (0.835,0.1);
		\coordinate (F2) at (0.666,0.333);
		\coordinate (F0) at (1.1664,0.434);
		\coordinate (F1) at (0.834,0.4338);
		\draw[dotted] (T0) -- (T2);
		% blue
		\fill[gray] (M) circle (0.4pt);
		\draw[gray,dashed] (M) -- (T0);
		\draw[gray,dashed] (M) -- (T1);
		\draw[gray,dashed] (M) -- (T2);
		\draw[gray,dashed] (M) -- (T3);
		
		\draw (T0) -- (T1) -- (T3)--(T0);
		\draw (T1) -- (T2) -- (T3);
		\fill (T0) circle (0.4pt);
		\fill (T1) circle (0.4pt);
		\fill (T2) circle (0.4pt);
		\fill (T3) circle (0.4pt);
	\end{scope}
	
	\begin{scope}[xshift = 2.1cm]
		
		% 0 - 2 - 3 (back)
		\begin{scope}[xshift = 0cm, yshift = 0.3cm]
			\coordinate (T0) at (0,0);
			\coordinate (T1) at (1,0);
			\coordinate (T2) at (1.5,0.3);
			\coordinate (T3) at (1,1);
			\coordinate (M) at (0.876,0.325);
			\coordinate (F3) at (0.835,0.1);
			\coordinate (F2) at (0.666,0.333);
			\coordinate (F0) at (1.1664,0.434);
			\coordinate (F1) at (0.834,0.4338);
			% red
			\draw[gray,dashed] (T0) -- (F1);
			\draw[gray,dashed] (T2) -- (F1);
			\draw[gray,dashed] (T3) -- (F1);
			\draw[gray,dashed] (M) -- (F1);
			\fill[gray] (F1) circle (0.4pt);
			% blue 
			\fill (M) circle (0.4pt);
			\draw (M) -- (T0);
			\draw (M) -- (T2);
			\draw (M) -- (T3); 
			
			\fill (T0) circle (0.4pt);
			\fill (T2) circle (0.4pt);
			\fill (T3) circle (0.4pt);
			\draw (T0) -- (T2) -- (T3) -- (T0);
			
		\end{scope}
		
		% 0 - 1 - 2 (base)
		\coordinate (T0) at (0,0);
		\coordinate (T1) at (1,0);
		\coordinate (T2) at (1.5,0.3);
		\coordinate (T3) at (1,1);
		\coordinate (M) at (0.876,0.325);
		\coordinate (F3) at (0.835,0.1);
		\fill[white] (T0) -- (T1) -- (T2);  
		\draw[dotted] (0,0) -- (1.5,0.3);     
		%red 
		\draw[gray,dashed] (T0) -- (F3);
		\draw[gray,dashed] (F3) -- (T2);
		\draw[gray,dashed] (F3) -- (M);
		\draw[gray,dashed] (F3) -- (T1);
		\fill[gray] (F3) circle (0.4pt);
		% blue
		\fill (M) circle (0.4pt);
		\draw (M) -- (T0);
		\draw (M) -- (T1);
		\draw (M) -- (T2);
		\draw (T0) -- (T1) -- (T2);
		
		\fill (T0) circle (0.4pt);
		\fill (T1) circle (0.4pt);
		\fill (T2) circle (0.4pt);
		
		% 0 - 1 - 3 (left)
		\begin{scope}[xshift = -0.3cm, yshift = 0.2cm]
			\coordinate (T0) at (0,0);
			\coordinate (T1) at (1,0);
			\coordinate (T2) at (1.5,0.3);
			\coordinate (T3) at (1,1);
			\coordinate (M) at (0.876,0.325);
			\coordinate (F3) at (0.835,0.1);
			\coordinate (F2) at (0.666,0.333);
			\fill[white] (T0) -- (T1) -- (T3);  
			% red
			\fill[gray] (F2) circle (0.4pt);
			% blue
			\draw[dotted] (M) -- (T0);
			\draw[dotted] (M) -- (T1);
			\draw[dotted] (M) -- (T3);
			\fill (M) circle (0.4pt);
			% red
			\draw[gray,dashed] (M) -- (F2);
			\draw[gray,dashed] (T0) -- (F2);
			\draw[gray,dashed] (T1) -- (F2);
			\draw[gray,dashed] (T3) -- (F2);
			
			\draw (T0) -- (T1) -- (T3)--(T0);  	
			\fill (T0) circle (0.4pt);
			\fill (T1) circle (0.4pt);
			\fill (T3) circle (0.4pt);
			
		\end{scope}
		
		% 1 - 2 - 3 (right)
		\begin{scope}[xshift = 0.35cm, yshift = 0.1cm]
			\coordinate (T0) at (0,0);
			\coordinate (T1) at (1,0);
			\coordinate (T2) at (1.5,0.3);
			\coordinate (T3) at (1,1);
			\coordinate (M) at (0.876,0.325);
			\coordinate (F3) at (0.835,0.1);
			\coordinate (F2) at (0.666,0.333);
			\coordinate (F0) at (1.1664,0.434);
			\coordinate (F1) at (0.834,0.4338);
			\fill[white] (T1) -- (T2) -- (T3) -- (M) -- (T1);  
			% blue
			\fill (M) circle (0.4pt);
			\draw (M) -- (T1);
			\draw[dotted] (M) -- (T2);
			\draw (M) -- (T3);
			% red
			\draw[gray,dashed] (T1) -- (F0);
			\draw[gray,dashed] (T2) -- (F0);
			\draw[gray,dashed] (T3) -- (F0);
			\draw[gray,dashed] (M) -- (F0);
			\fill[gray] (F0) circle (0.4pt);
			
			\fill (T1) circle (0.4pt);
			\fill (T2) circle (0.4pt);
			\fill (T3) circle (0.4pt);
			\draw (T1) -- (T2) -- (T3) -- (T1);
		\end{scope}
	\end{scope}
\end{tikzpicture}
\caption{Alfeld split into $4$ new tetrahedra (left) and  Worsey–Farin split into $3$ new tetrahedra each (right) in 3D with the new edges (dashed) and new vertices in gray. For simpler visualization the $4$ tetrahedra resulting from the Alfeld split are detached in the right-hand figure. 
}
      \label{fig:alfeldsplit}
\end{figure}
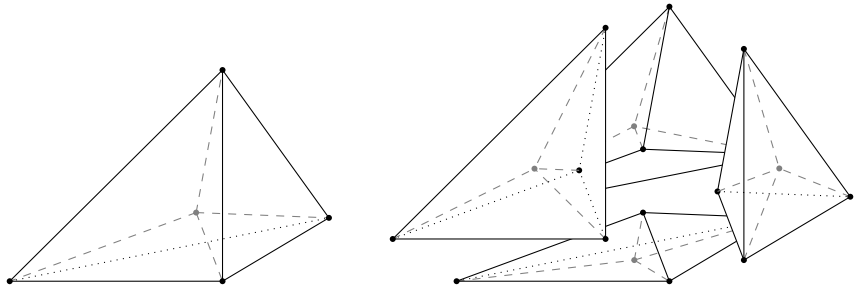
As above the following lemma is established by enumerating the mesh entities introduced by the split. 

\begin{lemma}[mesh quantities for 3D split meshes]\label{lem:split-meshes}
Let $\mesh$ be a conforming simplicial triangulation of a polyhedral bounded domain $\Omega \subset \mathbb{R}^3$ with $V$ vertices, $E$ edges, $F$ faces, and $T$ tetrahedra. 
Let $\meshA$ be the Alfeld split and $\meshW$ the Worsey--Farin split of $\mesh$.
\begin{enumerate}[label=(\roman*)]
\item \label{itm:split-A} $\meshA$ has $V_A = V+T$ vertices, $E_A = E + 4T$ edges, $F_A =F  + 6T$ faces, and $T_A = 4T$ tetrahedra. 
\item \label{itm:split-WF} $\meshW$ has $V_W = V+F+T$ vertices, $E_W = E + 8T + 3F$ edges,  $F_W = 3F + 18 T$ faces, and $T_W = 12T$ tetrahedra. 
\end{enumerate}
\end{lemma}

\section{Mesh quantities in 3D}
\label{sec:meshquantities}

Using counting arguments we want to express all mesh quantities in just a few quantifiable 
parameters, the values of which we intend to investigate and estimate.
In the following section we use those insights to compare the dimensions of mixed finite element spaces.  

The starting point are the well-known Euler characteristics 
both for the closed domain $\overline{\Omega}$ as well as for its boundary $\partial\Omega$, cf.~\cite{E.1758,H.2002}. 
Let $V$ be the number of vertices, $E$ the number of edges,  $F$ the number of faces, and $T$ the number of tetrahedra in $\mesh$. Furthermore, we denote the corresponding number of mesh quantities on the boundary by $V_b, E_b$, and $F_b$, respectively. 
Then, there exist constants $\chi, \chi_b \in \mathbb{Z}$ such that
\begin{align}\label{eq:Euler}
V - E + F - T = \chi \qquad \text{ and } \qquad 
V_b - E_b+ F_b  = \chi_b. 
\end{align}
The constants $\chi, \chi_b \in \mathbb{Z}$ are topological invariants that only depend on the genus of the domain $\overline{\Omega}$. 
For example, for $\overline{\Omega}$ simply connected one has $\chi = 1$ and $\chi_{b} = 2$. 
This can be shown by induction starting from a single tetrahedron.

Additionally, some simple combinatorial arguments show that 
\begin{align}\label{eq:topocef}
2F-F_b=4T  \qquad \text{ and } \qquad  3F_b=2 E_b. 
\end{align}
 The proof follows similarly as in the 2D case~\cite{lrsBIBaf}.
Imagine placing $2$ marbles on each interior face and one on each boundary face.
Now move the marbles into the neighboring tetrahedra.
Thus each tetrahedron will have $4$ marbles, since it has $4$ faces.
In total we have allocated $2F-F_b$ marbles to tetrahedra, and we end up with $4T$.
By the conservation of marbles, the first equality in \eqref{eq:topocef}
is proved.
The proof of the second identity follows analogously.
We put $3$ marbles on each boundary face, then move them to the boundary edges.
After the transfer each boundary edge has $2 $ marbles. 

\begin{remark}[2D case]\label{rmk:avg-2d}
In 2D one can express all mesh quantities $V,E,T$ ($T$ denotes the number of triangles) and $V_b, E_b$ in terms of only $V$ and $V_b$. 
Indeed, the Euler characteristics are 
\begin{align}\label{eq:Euler-2d}
V - E + T = \chi \qquad  \text{ and } \qquad 
V_b - E_b = \chi_b,
\end{align}
and the combinatorial identity 
\begin{align}\label{eq:topocef-2d}
%V_b =E_b,\qquad 
3 T =2 E - E_b. 
\end{align}
suffices. From those it follows that
\begin{alignat*}{2}
E &= 3V - V_b - 3 \chi + \chi_b, \qquad  \quad E_b &= V_b - \chi_{b},\\
T &= 2V - V_b - 2 \chi + \chi_b.  
\end{alignat*}
For a simply connected domain we further have that $\chi_b = 0$ and $\chi = 1$. 
\end{remark}

In 3D the identities~\eqref{eq:Euler} and~\eqref{eq:topocef} are not enough to express all mesh quantities in the number of vertices $V$ and $V_b$. In this case we need to choose an additional parameter. 
 
\subsection{Average number of edges}

Additionally to $V,V_b$ we consider the average number of edges intersecting in a vertex. 
In principle we could use alternative (average) quantities, cf.~Remark~\ref{rmk:altern} below, but we shall see in the sequel that this choice has some advantages. 

Let $e_i$ denote the number of edges emanating from vertex $v_i$ in $\mesh$, for $i=1,\ldots,V$. 
Then we define the arithmetic mean
\begin{align}\label{eqn:ebardef}
\ebar \coloneqq \frac1V\sum_{i=1}^V e_i = \frac{2E}{V},
\end{align}
where the second identity holds since summing $e_i$ over all $i$ counts every edge twice.
Thus $\ebar$ is the average number of edges emanating from vertices in a mesh $\mesh$. 

\begin{proposition}\label{pro:3d-quant}
Let $\mesh$ be a conforming simplicial triangulation with $V$ vertices, $E$ edges, $F$ faces, and $T$ tetrahedra of a domain $\Omega \subset \mathbb{R}^3$ as above with topological constants $\chi,\chi_b$, cf.~\eqref{eq:Euler}. 
Then, the mesh quantities can be expressed in terms of $V, V_b$, and $\ebar$, as
\begin{alignat*}{2}
	E & = \tfrac{1}{2}\ebar V, \qquad\qquad &E_b &= 3(V_b - \chi_b), \\
	F & = \left(\ebar - 2 \right)V  - V_b + 2\chi +  \chi_b, \qquad \quad &F_b&= 2(V_b - \chi_b),\\
	T & = \left(\tfrac{1}{2} \ebar - 1 \right)V  -  V_b + \chi + \chi_b.   
\end{alignat*}
\end{proposition}
\begin{proof}
The expression for $E$ follows by the definition of $\ebar$. 
From~\eqref{eq:topocef} we have that $F_b = \tfrac{2}{3}E_b$ and applying it in the equation for the boundary quantities in~\eqref{eq:Euler} yields
\begin{align*}
E_b = 3(V_b - \chi_b) \quad \text{ and } \quad 
F_b =  2(V_b - \chi_b). 
\end{align*}
Inserting the expression for $F_b$ in the first identity in~\eqref{eq:topocef} and using the expression of $\ebar$ in the first identity in~\eqref{eq:Euler} leads to 
\begin{align*}
	F - 2T &= V_b - \chi_b,\\
	F - T &= \chi - V + \tfrac{1}{2} \ebar V. 
\end{align*}
Solving the system of equations for $F$ and $T$ proves the identities as stated. 
\end{proof}

\begin{remark}\label{rmk:3d-approx-quant}
If for a fixed domain $\Omega \subset \mathbb{R}^3$ the mesh quantities in the mesh $\mesh$ are sufficiently large one may neglect the asymptotically smaller quantities $\chi, \chi_b$, and $V_b$. Then, one arrives at 
\begin{align}\label{eq:approx-quant}
E  = \tfrac{1}{2}\ebar V, \qquad 
F  \approx (\ebar-2)V, \qquad \text{ and } 
\quad
T  \approx \left(\tfrac{1}{2} \ebar-1\right)V.
\end{align}
We use $\approx$ for families of meshes for which $T$, $F$, $E$, and $V$ all increase without bound, and the ratio $V_b/V$ converges to zero.
However, the asymptotics in most cases are just a guide for the reader
to avoid carrying more terms that can confuse the issue. 
Note that one can easily trace
the estimates to obtain exact equalities with the terms restored that have been
eliminated as being lower order. 
\end{remark}

Before investigating general bounds on $e_i$ and $\ebar$ let us consider the values for special  meshes. 

\begin{example}[star meshes]\label{ex:small-mesh} 
We consider a simply connected open $\Omega\subset \mathbb{R}^3$ and meshes with a single interior vertex $v_1$. 
In this situation $\partial \Omega$ is a polyhedral surface topologically equivalent to  a 2-sphere $(\chi = 1, \chi_b = 2)$ and the one interior vertex is connected to each vertex on the boundary, i.e., one has that 
\begin{align*}
	V_b = e_1,  \qquad V = V_b + 1 \quad \text{ and } \quad E = E_b + V_b.  
\end{align*} 
These triangulations can be viewed as the star of an interior vertex for
general triangulations, and quantities associated with them will be relevant
in subsequent derivations. 

The simplest example is the domain and mesh $\mesh$ arising from the Alfeld split of one single tetrahedron. 
In this case we have that $V_b = 4 = e_1$, $V = 5$, and $E = 10$. 

This is indeed the smallest such mesh in the sense that  general star meshes of the type as described satisfy $e_1 = V_b \geq 4$, and $E\geq 10$ as we shall show now.
Using the fact that each boundary vertex has at least $3$ adjacent boundary edges and there is at least one boundary vertex it follows that $F_b \geq 3$. 
From the expressions in Proposition~\ref{pro:3d-quant} we have that  $F_b=2(V_b-2)$ and $E_b=3(V_b-2)$, since $\chi_b = 2$.  
Hence, it follows that $V_b=F_b/2 +2>3$ and thus $V_b \geq 4$, since $V_b \in \mathbb{N}$.  
Then it also follows that $E_b=3(V_b-2)\geq 6$ and that $E =  E_b + V_b \geq 10$.

For any given polyhedral domain $\Omega$ with positive measure, the number $e_1 = V_b$ can be
arbitrarily large for suitably chosen meshes. 
Note however, that for increasing $V_b$ the shape-regularity of the mesh deteriorates. 
Indeed, $e_1$ can be bounded above by the shape-regularity constant or minimum solid angle, as stated in Lemma~\ref{lem:upper-bd} below. 

However, $\ebar$ cannot be arbitrarily large, since we take the mean over all vertices. 
Recalling that $ V = V_b + 1$, $E = E_b + V_b$, and that $E_b = 3(V_b - 2)$, we obtain 
\begin{align*}
\overline{e} = \frac{2E}{V} = 
%2 (E_b + V_b) = 2(3(V_b - 2) + V_b)  = 8 V_b - 12 
\frac{8V_b - 12}{V_b + 1}. 
\end{align*}
This expression is strictly growing in $V_b$ and bounded above by $8$. 
Since we have $V_b \geq 4$ by the above arguments, we find that 
\begin{align*}
4 \leq  \overline{e} \leq 8.
\end{align*}
The boundary vertices with fewer edges emanating from them are responsible for $\overline{e}$ being bounded. 
\end{example}

\begin{example}[regular meshes]\label{ex:freudenthal}
\begin{figure}
\centering
     \begin{subfigure}[b]{0.48\textwidth}
         \centering
         \begin{tikzpicture}[scale=2.8]
	
	\begin{scope}
		
		\coordinate (T0) at (0,0);
		\coordinate (T1) at (1,0);
		\coordinate (T2) at (1.5,0.2);
		\coordinate (T3) at (1.5,1.2);
		\coordinate (T4) at (0.5,1.2);
		\coordinate (T5) at (0,1);
		\coordinate (T6) at (1,1);
		\coordinate (T7) at (0.5,0.2);

		\draw (T0) -- (T1) -- (T2) -- (T3) -- (T4) -- (T5) -- (T0);
		\draw[dashed] (T0) -- (T7) -- (T4);
		\draw[dashed] (T7) -- (T2);
		%back edges
		\draw[gray,dashed] (T0) -- (T4);   
		\draw[gray,dashed] (T0) -- (T2);   
		\draw[gray,dashed] (T7) -- (T3); 
		\draw[gray,dashed] (T0) -- (T3);

		\draw  (T1) -- (T6) -- (T3);
		\draw  (T5) -- (T6) ;
		
		%fr'ont edges
		\draw[gray] (T0) -- (T6);   
		\draw[gray] (T1) -- (T3);   
		\draw[gray] (T5) -- (T3);   
		
	\end{scope}
	
\end{tikzpicture}
         \caption{Kuhn triangulation of a single cube into 6 tetrahedra~\cite{K.1960}.}
         \label{fig:Kuhncube}
     \end{subfigure}
     \hfill
          \begin{subfigure}[b]{0.48\textwidth}
         \centering

\begin{tikzpicture}[scale=1.8]
	
	\begin{scope}
		
		\coordinate (T0) at (0,0);
		\coordinate (T1) at (1,0);
		\coordinate (T2) at (1.5,0.2);
		\coordinate (T3) at (1.5,1.2);
		\coordinate (T4) at (0.5,1.2);
		\coordinate (T5) at (0,1);
		\coordinate (T6) at (1,1);
		\coordinate (T7) at (0.5,0.2);
		
		\draw (T0) -- (T1) -- (T2) -- (T3) -- (T4) -- (T5) -- (T0);
		\draw[dashed] (T0) -- (T7) -- (T4);
		\draw[dashed] (T7) -- (T2);
		%back edges
		\draw[dashed] (T0) -- (T4);   
		\draw[dashed] (T0) -- (T2);   
		\draw[gray,dashed] (T7) -- (T3); 
		\draw[gray,dashed] (T0) -- (T3);

		\draw  (T1) -- (T6) -- (T3);
		\draw  (T5) -- (T6) ;
		
		%fr'ont edges
		\draw (T0) -- (T6);   
		\draw[gray] (T1) -- (T3);   
		\draw[gray] (T5) -- (T3);   
		
		\fill[gray] (T3) circle (0.4pt);
		
		\draw[gray] (T3) -- (T2);   
		\draw[gray] (T3) -- (T6);   
		\draw[gray] (T3) -- (T4);   
		
	\end{scope}

	\begin{scope}[xshift = 1.5cm,yshift = 1.2cm]
		
		\coordinate (T0) at (0,0);
		\coordinate (T1) at (1,0);
		\coordinate (T2) at (1.5,0.2);
		\coordinate (T3) at (1.5,1.2); 
		\coordinate (T4) at (0.5,1.2);
		\coordinate (T5) at (0,1);
		\coordinate (T6) at (1,1);
		\coordinate (T7) at (0.5,0.2);
		
		\draw (T0) -- (T1) -- (T2) -- (T3) -- (T4) -- (T5) -- (T0);
		\draw[dashed] (T0) -- (T7) -- (T4);
		\draw[dashed] (T7) -- (T2);
		%back edges
		\draw[gray,dashed] (T0) -- (T4);   
		\draw[gray,dashed] (T0) -- (T2);   
		\draw[dashed] (T7) -- (T3); 
		\draw[gray,dashed] (T0) -- (T3);

		\draw  (T1) -- (T6) -- (T3);
		\draw  (T5) -- (T6) ;
		
		%fr'ont edges
		\draw[gray] (T0) -- (T6);   
		\draw (T1) -- (T3);   
		\draw (T5) -- (T3);   
		
		\fill[gray] (T0) circle (0.4pt);
		
		\draw[gray] (T0) -- (T1);   
		\draw[gray] (T0) -- (T6);   
		\draw[gray] (T0) -- (T5);   
		
	\end{scope}
	
	\begin{scope}[xshift = 1.6cm, yshift = -0.1cm]
		
		\coordinate (T0) at (0,0);
		\coordinate (T1) at (1,0);
		\coordinate (T2) at (1.5,0.2);
		\coordinate (T3) at (1.5,1.2);
		\coordinate (T4) at (0.5,1.2);
		\coordinate (T5) at (0,1);
		\coordinate (T6) at (1,1);
		\coordinate (T7) at (0.5,0.2);

		\draw (T0) -- (T1) -- (T2) -- (T3) -- (T4) -- (T5) -- (T0);
		\draw[dashed] (T0) -- (T7) -- (T4);
		\draw[dashed] (T7) -- (T2);
		%back edges
		\draw[dashed] (T0) -- (T4);   
		\draw[dashed] (T0) -- (T2);   
		\draw[dashed] (T7) -- (T3); 
		\draw[dashed] (T0) -- (T3);

		\draw  (T1) -- (T6) -- (T3);
		\draw  (T5) -- (T6) ;
		
		%fr'ont edges
		\draw (T0) -- (T6);   
		\draw (T1) -- (T3);   
		\draw (T5) -- (T3);   
		
	\end{scope}
	
\end{tikzpicture}
         \caption{Two Kuhn cubes sharing a vertex, together with a third Kuhn cube ready
to be attached.}
         \label{fig:Kuhntria}
     \end{subfigure}
     \label{fig:kuhnt}
     \caption{Regular triangulation of a cube and combination of those to form the triangulation of a rectangular domain.}
\end{figure}
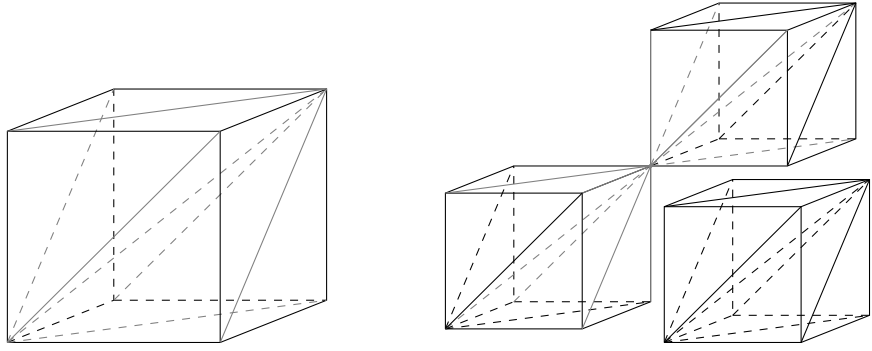
     
     The so-called Kuhn triangulation~\cite{K.1960,ref:Freudenthalulation}
    is a regular triangulation of the unit cube in $\mathbb{R}^d$. For dimension $d = 3$ it consists of $6$ tetrahedra. 
 In this case it has one interior edge and, for each vertex, it has 3 axis-parallel edges
and 3 edges lying on faces, as indicated in Figure~\ref{fig:Kuhncube}.
Thus there are 7 edges impinging on each of the vertices at the lower-left and
upper-right of the cube. 

We refer to a triangulation consisting of Kuhn cubes as \emph{Freudenthal triangulation}, cf.~\cite{F.1942, ref:Freudenthalulation}. 

\begin{enumerate}[label=(\alph*)]
\item For an infinite mesh of the whole of $\mathbb{R}^3$ arising by filling the space with Kuhn cubes for every vertex, one has that $e_i = 14$. 
Indeed, in Figure~\ref{fig:Kuhntria} we see two Kuhn cubes sharing a vertex, with a third cube in position to be added to the complex. 
In adding the new cube, no new edges are added, and this holds for all six positions where a cube is added to form a space filling triangulation. 
All edges of the shared vertex are contained already in the two diagonally opposite Kuhn cubes and thus we have that $e_i = 2 \cdot 7 = 14$, for all $i$. 
\end{enumerate}
The quantity $\ebar$ is defined only for finite meshes, so we now consider
finite subsets of the Freudenthal triangulation based on Kuhn cubes.
We can then talk about a limiting value for $\ebar$ as the number of cubes
tends to infinity.
\begin{enumerate}[label=(\alph*),resume]
\item For a bounded rectangular domain $\Omega \subset \mathbb{R}^3$ and a regular mesh thereof consisting of $n \cdot m \cdot \ell$ Kuhn cubes we have $e_i = 14$ only for interior vertices. 
The boundary vertices have fewer adjacent edges, and hence one has that $\ebar < 14$. 
For $m, n, \ell \to \infty$ it follows that $\ebar \to 14$, since the number of boundary edges and vertices are asymptotically smaller than the number of interior edges and vertices. 
\end{enumerate} 
These arguments can easily be generalized to higher dimension $d$, where a Freudenthal triangulation consisting of $N^d$ Kuhn cubes has order $N^d$ vertices and order  $(2^d-1)N^d$ edges. 
Hence the corresponding value for $\ebar$ is $2(2^d-1)$ in $d$ dimensions cf.~\cite[Sec.~3.1]{DST.2022}. 
\end{example}

In view of the following remark one could chose any other average quantity instead of $\ebar$ as defined in~\eqref{eqn:ebardef} to fully determine all other parameters.
However, we prefer to work with the average mesh quantity determined by the lowest dimensional quantities, i.e., number of vertices and number of edges. 
In terms of data structure they can easily be extracted. 

\begin{remark}[alternative quantities]\label{rmk:altern}
	For an interior vertex $v_i$ in $\mesh$ we can relate the number $e_i$ of edges emanating from $v_i$ to the number $t_i$ of simplices intersecting in $v_i$ and the number of faces $f_i$ intersecting in $v_i$. 
	For this it suffices to consider the neighboring patch, the star of the vertex $v_i$ in $\mesh$, defined by
	\begin{align}\label{def:star}
		\omega(v_i) \coloneqq \{\tet \in \mesh \colon v_i \in \tet\}. 
	\end{align} 
Then $e_i$ is the number of boundary vertices of $\partial \omega(v_i)$
and $f_i, t_i$ are the number of the respective boundary quantities 
(edges and triangles) of $\partial \omega(v_i)$.
With the same arguments as in Proposition~\ref{pro:3d-quant} and Example~\ref{ex:small-mesh} we find that 
	\begin{align}\label{eq:tetperv}
		f_i =  3(e_i - 2)  \quad  \text{ and } \quad 
		t_i = 2(e_i - 2), 
	\end{align}
	for any interior vertex $v_i$. 
	For the global average quantities we obtain
\begin{align}
		\overline{f} \coloneqq \frac{1}{V} \sum_{i = 1}^V f_i = \frac{3F}{V}  \quad  \text{ and } \quad 
		\overline{t} \coloneqq \frac{1}{V} \sum_{i = 1}^V t_i = \frac{4T}{V}. 
	\end{align}
	With Proposition~\ref{pro:3d-quant} neglecting the boundary quantities and both $\chi$ and $\chi_b$ for large meshes, we find that 
	\begin{align}\label{eq:10}
		\overline{f} \approx 3 (\overline{e} - 2)  \quad  \text{ and } \quad 
		\overline{t} \approx 2 (\overline{e} - 2). 
	\end{align}

	Also the remaining average quantities can be recovered. 
	Denoting by $\theta_j$ the number of tetrahedra intersecting in edge $\edgnam_j$ and 
by $\phi_j$ the number of faces intersecting in edge  $\edgnam_j$, 
for $j = 1, \ldots, E$, we may define the average quantities 
	\begin{align}
		\overline{\phi} \coloneqq \frac{1}{E} \sum_{j = 1}^E \phi_j = \frac{3F}{E}  \quad  \text{ and } \quad 
		\overline{\theta} \coloneqq \frac{1}{E} \sum_{j = 1}^E \theta_j = \frac{6T}{E}. 
	\end{align}
They can be recovered from $\ebar$ by 
	\begin{align}\label{eq:12}
		\overline{\phi} = 2 \frac{\overline{f}}{\overline{e}}  \approx 6 \frac{(\ebar-2)}{\ebar}  \quad  \text{ and } \quad 
		\overline{\theta} = 3 \frac{\overline{t}}{\overline{e}} \approx 6 \frac{(\ebar-2)}{\ebar}, 
	\end{align}
	where the approximate identity refers again to large meshes where we may neglect the boundary quantities, and both $\chi$ and $\chi_b$. 
See also~\cite[Thm.~3.3]{PR.2005}. 
\end{remark}

\begin{example}
For the regular mesh in Example~\ref{ex:freudenthal}, from $\ebar = 14$ the remaining average
quantities can be deduced using Remark~\ref{rmk:altern} since the approximations in~\eqref{eq:10} and~\eqref{eq:12} become exact on an infinite mesh, as summarized in Table~\ref{tbl:reg}.

	\begin{table}[ht!] \small
	\begin{TAB}(r)[5pt]{|c|c|c|c|c|}{|c|c|}
		$\ebar$ & $\overline{f}$ & $\overline{t}$ & $\overline{\phi}$ & $\overline{\theta}$ 
		\\
		$14$ &  $36$ &   $24$ &   $\frac{36}{7}$ &   $\frac{36}{7}$  	
	\end{TAB}
	\caption{ Average mesh quantities for the regular mesh as in Example~\ref{ex:freudenthal}.  
	}
	\label{tbl:reg}
\end{table}
\end{example}

\begin{remark}[Average quantities in 2D]\label{rmk:2d-av-quant}
In 2D the quantities $\ebar, \tbar$ are bounded as
\begin{align*}
	2 \leq \ebar \leq 6 \qquad \text{ and } \qquad 1 \leq \tbar \leq 6, 
\end{align*}  
and for any sequence of meshes $(\mesh_n)_{n \in \mathbb{N}}$ of a given polyhedral domain $\Omega$ with positive measure 
$\Omega \subset \mathbb{R}^2$ with $V_n \to \infty$, as $n \to \infty$ one has that 
\begin{align*}
	\ebar_n \to 6 \quad \text{ and }\quad \tbar_n \to 6, \qquad \text{ as } n \to \infty.  	
\end{align*}
This is a direct consequence of Euler characteristics~\eqref{eq:Euler-2d}, since  
\begin{align*}
	\ebar  = \frac{2E}{V} = 6 - \frac{ 2V_b}{V} -  \frac{6 \chi}{V} + \frac{ 2 \chi_b}{V},\qquad
	\tbar  = \frac{3T}{V} = \frac{2E-E_b}{V}=\ebar-\frac{E_b}{V},
\end{align*}
which holds for any triangulation in the plane, cf.~Remark~\ref{rmk:avg-2d}. 
In the formulation for simple planar graphs this result can be found, e.g., in~\cite[Thm.~6.1.23]{W.1996}. 
Note that the sequence $(\mesh_n)_{n \in \mathbb{N}}$ may arise as an arbitrary refinement of an initial mesh $\mesh_0$ and no assumption of shape-regularity or specific refinement scheme enters here. 
		
See also~\cite[Thm.~5.1]{PR.2002}, which states this for the special case of so-called skeleton-regular refinements. 
\end{remark}

In 3D to obtain bounds on the average quantities and their asymptotic behavior is much more challenging.  
In fact, different limiting values of $\ebar$ may occur depending on the mesh refinement scheme, cf.~\cite[Tab.~3]{PR.2005} and the investigation below. 

\subsection{Lower bounds}
Let us first consider lower bounds for $e_i$ and $\ebar$. 
Obviously, each vertex $e_i$ is contained in at least one tetrahedron, and hence $e_i \geq 3$. 
This implies the first lower bound $\ebar \geq 3$. However, this can be improved in most cases. 

\begin{lemma}\label{lem:lower-bd}
Let $\mesh$ be a conforming simplicial triangulation of an open bounded domain $\Omega \subset \mathbb{R}^3$ as above. 
Then, for any interior vertex $v_i \notin \partial \Omega$ the number $e_i$ of edges intersecting at the vertex $v_i$ is bounded below as
\begin{align*}
e_i \geq 4. 
\end{align*}
Furthermore, if $\Omega$ is pathwise connected and $\mesh$ has at least one interior vertex, then we have that
\begin{align}\label{est:lower-bd}
\ebar \geq 4. 
\end{align}
\end{lemma}
\begin{proof} 
The estimate $e_i \geq 4$ for any interior vertex $v_i$ follows by Example~\ref{ex:small-mesh}. 

Since $e_i \geq 3$ in general and $e_i \geq 4$ for interior vertices $v_i$ the only vertices that might hamper the lower bound~\eqref{est:lower-bd} are boundary vertices $v_i$ with $e_i = 3$. 
For such vertices we can imagine successively removing the unique tetrahedron $\tet_i \in \mesh$ with $v_i \in \tet_i$ from the mesh.  
 Since the open domain $\Omega$ is pathwise connected, every tetrahedron shares at least one face ($3$ vertices and $3$ edges) with the remaining mesh.  
Hence, by removing $\tet_i$ we remove $4-3 = 1$ vertex and $ 6- 3 = 3$ edges. 
After removing $\tet_i$, the remaining mesh $\mesh_1 \coloneqq \mesh \setminus \{\tet_i\}$ still covers a pathwise connected closed domain, and hence we may proceed inductively. 
Indeed, the tetrahedron $\tet \in \mesh$ that shared a face with $\tet_i$ shares at least another face with the remaining mesh $\mesh \setminus \{\tet_i, \tet\}$, because otherwise $\Omega$ would not be pathwise connected. 

Note that this inductive process terminates with a simplicial mesh $\mesh' \subset \mesh$, since the number of vertices in $\mesh$ is bounded. 
The triangulation $\mesh'$ is not empty, since $\mesh$ has at least one interior vertex and since in this procedure no interior vertex becomes a boundary vertex. 
 By construction any vertex in $\mesh'$ satisfies that $e_i \geq 4$, and thus $\ebar' \geq 4$. 
Denoting by $k$ the number of tetrahedra removed, and hence the number of vertices removed, the triangulation $\mesh' $ has $V' = V -k$ vertices and $E' = E - 3k$ edges. 
Then, it follows that 
\begin{align*}
4 \leq \ebar' = \frac{2E'}{V'} = \frac{2 E - 6k}{V - k} = 
\frac{\ebar V - 6k}{V -k}.
\end{align*}
Rearranging leads to 
\begin{align*}
	\ebar \geq 4 + \frac{2k}{V} \geq 4,
\end{align*}
which proves the claim. 
\end{proof}

\begin{remark} 
	The proof also works, if $\overline{\Omega}$ is simply connected, in which case there might be degeneracies in the domain and in the mesh in the following way: There are two parts of the triangulation, that are connected only through a point or only through an edge. 
However, for the finite element setting such a situation is of lesser relevance and hence we refrain from presenting the proof. 
\end{remark}

Since the lower bound $4$ is attained by an Alfeld split of a single tetrahedron the estimate~\eqref{est:lower-bd} is sharp. 

\subsection{Upper bounds}
It is possible to obtain bounds on $\ebar$ using the mesh shape-regularity. 
There are several equivalent measures for this~\cite{brandts2008equivalence}, however using the minimal solid angle allows for a very simple proof. 
By the equivalence upper bounds in terms of other shape-regularity constants can be derived. 
 
For any tetrahedron $\tet \in \mesh$, we denote by $\alpha_{\tet,i} \in (0,2\pi)$ the solid angle of  $\tet$ with vertex $v_i \in \tet$ as apex, and the minimal solid angles at $v_i$ and in $\mesh$, respectively, by 
\begin{align*}
	\alpha_i \coloneqq  \min_{\tet \in \omega(v_i)} \alpha_{\tau,i} \quad \text{ and } \quad 
	\alpha_\mesh \coloneqq  \min_{i = 1, \ldots, V} \alpha_{i}. 
\end{align*}
With $t_i$ the number of tetrahedra sharing a vertex $v_i$ as introduced in Remark~\ref{rmk:altern}, we may consider the mean solid angle $\overline{\alpha}_i$ at $v_i$
\begin{align}\label{est:angles-1}
	\overline{\alpha}_i \coloneqq\frac{1}{t_i} \sum_{\tet \in \omega(v_i)} \alpha_{\tau,i}  \geq \alpha_{i} \geq \alpha_{\mesh}. 
\end{align}
For an interior vertex $v_i$ it satisfies
\begin{align}\label{est:angles-2}
	\overline{\alpha}_i=  \frac{4\pi}{t_i}. 
\end{align}
This means that for interior vertices $t_i$ and $\overline{\alpha}_i$ contain the same amount of information and by \eqref{eq:tetperv} and so do $e_i$ and $f_i$. 

\begin{lemma}\label{lem:upper-bd}
Let $\mesh$ be a conforming simplicial mesh of a domain $\Omega \subset \mathbb{R}^3$ as above. 
We assume that $\mesh$ is shape-regular with mean solid angle 
$\overline{\alpha}_i$ and minimal solid angles $\alpha_i, \alpha_{\mesh}\in (0,2\pi)$ for $i = 1, \ldots, V$ as defined above. 
Then, for any 
interior vertex $v_i$ in $\mesh$
 the number $e_i$ of edges intersecting at the vertex $v_i$ is bounded  by 
\begin{align*}
e_i 
= 
2 + \frac{2 \pi}{\overline\alpha_i}
\leq 2 + \frac{2 \pi}{\alpha_i} 
\leq
 2 + \frac{2 \pi}{\alpha_{\mesh}}.  
\end{align*}
\end{lemma}

\begin{proof} 
	Let $v_i$ be an interior vertex. 
By definition of the the mean solid angle $\overline{\alpha}_i$ at $v_i$ as in \eqref{est:angles-1}, and with \eqref{est:angles-2} we have that 
\begin{align*}
t_i = \frac{4 \pi}{\overline\alpha_i} \leq \frac{4 \pi}{\alpha_{i}} \leq \frac{4 \pi}{\alpha_{\mesh}}. 
\end{align*}
Using~\eqref{eq:tetperv} for any interior vertex $v_i$ in $\mesh$ this shows
\begin{align*}
e_i = \tfrac{1}{2} t_i + 2 
\leq \frac{2 \pi}{\overline\alpha_i} + 2 
\leq \frac{2 \pi}{\alpha_i} + 2 
\leq \frac{2 \pi}{\alpha_\mesh} + 2,
\end{align*}
which proves the claim. 
\end{proof}

For the Freudenthal mesh, each tetrahedron has only two different solid angles,
$\pi/4$ and $\pi/12$, with each angle appearing for two vertices
of each tetrahedron.
This can easily be demonstrated using a representation of the tetrahedra
in the Freudenthal triangulation (a.~k.~a.~Kuhn cube), given e.g.~in
\cite{lrsBIBiu}, using standard formul{\ae} for the solid angle. 
Then, since the triangulation is regular, one has that 
\begin{align*}
	\alpha_i = \alpha_{\mesh} = \frac{\pi}{12}. 
\end{align*}
This results in the bound 
\begin{align*}
e_i \leq 26. 
\end{align*}
Note that such estimates do not serve practical purposes, since evaluating $\ebar = \frac{2E}{V}$ is still simpler. 
However, if by some refinement scheme that preserves the shape regularity, one has, e.g., the minimum solid angle at hand, then it gives a simple estimate. 

\begin{remark} 
For $d$-dimensional simplicial complexes with $f_i$ faces of dimension $i = 0, \ldots, d$ the $f$-vector of the complex is defined as $(f_{-1},f_0, \ldots, f_d) \in \mathbb{N}^{d+2}$ with the convention $f_{-1} = 0$, cf.~\cite[Def.~8.16]{Ziegler1995}, or~\cite{S.1992}. 
Note that simplicial triangulations in the context of numerical analysis are nothing else but pure simplicial complexes, i.e., a simplicial complex for which any $n$-face with $n \leq d$ is a subset of a $d$-simplex. 
This means that results on bounds of the $f$-vectors directly translate to our setting, with $\ebar = \frac{2 f_1}{f_0}$. 
For general simplicial complexes the theorem by Sperner~\cite{S.1928}, cf.~\cite[Lem.~8.29]{Z.1995} yields the upper bound
\begin{align*}
	\ebar \leq V-1. 
\end{align*}
This is too coarse an estimate for our needs. 
It may well be that in the fields of graph theory or of algebraic topology better estimates on $\ebar$ are available but unknown to the authors. 
In general results on $f$-vectors of simplicial complexes for numerically relevant (possibly even adaptive) refinement would be of great interest. 
\end{remark}

\subsection*{Asymptotic behavior under uniform mesh refinement} 
In Example~\ref{ex:freudenthal} we have seen that the value $\ebar = 14$ is attained for certain regular meshes. 
However, the values of the average quantities attained for the regular mesh also appear as limiting values for families of uniformly refined meshes for a large class of refinement algorithms, as proved in a general setting in~\cite[Thm.~3.6]{PR.2005}. 
See also~\cite{PR.2003.c} for specific refinement schemes. 
By \emph{uniform refinement} (opposed to adaptive refinement) we mean that each simplex in the triangulation is refined according to the refinement algorithm under consideration. 

For the sake of clarity we shall review the results on the asymptotic behavior of $\ebar$ under uniform mesh refinement in Lemma~\ref{lem:ebar-asymp} below. 
This is solely based on techniques in~\cite{PR.2005} and requires only minimal modification of the arguments. 
Afterwards in Example~\ref{ex:refinement} below we discuss examples of refinement algorithms that satisfy the assumptions.  
\smallskip 

We consider a \emph{uniform mesh refinement scheme} $\mathcal{R}$, which when applied to a simplicial conforming mesh $\mesh$ generates a simplicial conforming mesh $\mesh' = \mathcal{R}(\mesh)$ and has the following properties: 
\begin{enumerate}[label=(R\arabic*)]
	\item  \label{itm:R1} each tetrahedron in $\mesh$ is split into $8$ tetrahedra, each face is split into $4$ faces and each edge is split into $2$ edges;
	\item \label{itm:R2} the vertices in $\mesh'$ are exactly the vertices in $\mesh$ and all the mid points of edges in $\mesh$.
\end{enumerate} 

\begin{lemma}[cf.~{\cite[Thm.~3.6]{PR.2005}}]\label{lem:ebar-asymp}
Let $\Omega\subset\mathbb{R}^3$ be an open domain as above with $\mesh_0$ a conforming tetrahedral triangulation of $\Omega$. 
We assume that $\mathcal{R}$ is a (uniform) refinement scheme satisfying~\ref{itm:R1} and~\ref{itm:R2}. 
Then, for the family of meshes $\{\mesh_n\}_{n \in \mathbb{N}}$ generated by $\mesh_{n} = \mathcal{R}(\mesh_{n-1})$ for $n \geq 1$, one has that
 \begin{align*}
 \ebar_n \to 14, \quad  \text{ as }n \to \infty.
 \end{align*}  
\end{lemma}
\begin{proof}
The proof follows~\cite[Thm~3.6]{PR.2005} where the asymptotic behavior of $\overline{\theta}$ and $\tbar$ is investigated.  
For the sake of readability for our special case we present a slightly simpler version of the proof than the one therein. 
We consider a mesh $\mesh$ with $V$ vertices, $E$  edges, $F$ faces, and $T$ tetrahedra. 
Let $\mesh' = \mathcal{R}(\mesh)$ be the refined mesh with mesh quantities  $T',F',E'$, and $V'$.  
Then by the properties~\ref{itm:R1},~\ref{itm:R2} we obtain 
 \begin{align}\label{eq:refine-1}
 	T'  = 8 T,\quad
 	V'  =  V + E,\quad
	E_b'=2E+3F, \quad \text{ and } \quad 
	F_b'=4F.
 \end{align}
For the mesh consisting of a single tetrahedron, i.e., with $\widetilde{T} = 1$, $\widetilde{V} = 4$, $\widetilde{E} = 6$ and $\widetilde{F} = 4$ this means that after the refinement one has
 \begin{align}\label{eq:refine-2}
	\widetilde{T}'  = 8 ,\quad
	\widetilde{V}'  =  10,\quad
	\widetilde{E}_b'= 24, \quad \text{ and } \quad 
	\widetilde{F}_b'=16.
\end{align}
Applying~\eqref{eq:topocef} and \eqref{eq:Euler} with $\chi = 1$ we find that 
\begin{align*}
\widetilde{F}' = \widetilde{2} T' + \frac{1}{2} \widetilde{F}_b' = 24
 \quad \text{ and }  \quad 
\widetilde{E}' = \widetilde{V}' + \widetilde{F}' - \widetilde{T}' - 1 = 25.	
\end{align*}
This, in turn, implies that the numbers of newly generated interior faces and interior edges inside that one tetrahedron are 
\begin{align*}
\widetilde{F}' - \widetilde{F}_b' = 24-16 = 8
 \quad \text{ and }  \quad 
\widetilde{E}'-\widetilde{E}_b' = 25-24 = 1.
\end{align*}
This happens in every tetrahedron of a general mesh with $T$ tetrahedra. 
Thus, we obtain that after the uniform refinement one has
 \begin{align}\label{eq:refine-2}
 F' = 4F + 8T \quad \text{ and } \quad 
 E' = 2E + 3F + T.
 \end{align}
This means that 
\begin{align*}
\begin{pmatrix} 
V' \\ E' \\ F' \\ T'
\end{pmatrix} = 
\underbrace{\begin{pmatrix} 
1 & 1 & 0 & 0\\
0 & 2 & 3 & 1\\
0 & 0 & 4 & 8\\
0 & 0 & 0 & 8
\end{pmatrix}}_{\eqqcolon A}
\begin{pmatrix} 
V \\ E \\ F \\ T
\end{pmatrix}. 
\end{align*}
Then, for a sequence of meshes $(\mesh_n)_{n \in \mathbb{N}}$ generated by $\mesh_n = \mathcal{R}(\mesh_{n-1})$, for $n \in \mathbb{N}$ starting from $\mesh_0$, we find 
\begin{align*}
\begin{pmatrix} 
V_n \\ E_n \\ F_n \\ T_n
\end{pmatrix} = 
A^n
\begin{pmatrix} 
V_0 \\ E_0 \\ F_0 \\ T_0
\end{pmatrix},
\end{align*}
where $V_n, E_n, F_n, T_n$ are the respective mesh quantities in $\mesh_n$, for $n \in \mathbb{N}$. 
By symbolic calculation, e.g., with the Matlab toolbox~\cite{matlabsymbolic} we can express the iteration matrix as $A = V D V^{-1}$ with 
\begin{align*}
D = \begin{pmatrix} 
1 & 0 & 0 & 0\\
0 & 2 & 0 & 0\\
0 & 0 & 4 & 0\\
0 & 0 & 0 & 8
\end{pmatrix}
 \quad \text{ and } \quad 
V = \begin{pmatrix} 
1 & 1 & 1/2 & 1/6\\
0& 1& 3/2& 7/6\\
0& 0&   1&  2\\
0& 0&   0&   1
\end{pmatrix}. 
\end{align*}
Then the inverse of $V$ is given by 
\begin{align*}
V^{-1} = \begin{pmatrix} 
1& -1&    1&   -1\\
0&  1& -3/2& 11/6\\
0&  0&    1&   -2\\
0&  0&    0&    1
\end{pmatrix}. 
\end{align*}
Thus, we obtain that
\begin{align*}
A^n  &=(V D V^{-1})^n = V D^n V^{-1} \\
& = 
\begin{pmatrix}
1&  2^n - 1&  \tfrac{1}{2} 4^n- \tfrac{3}{2} 2^n + 1&   \tfrac{11}{6} 2^n - 4^n + \tfrac{1}{6} 8^n - 1\\
0&     2^n& \tfrac{3}{2}4^n - \tfrac{3}{2} 2^n & \tfrac{11}{6} 2^n - 3\cdot 4^n + \tfrac{7}{6}8^n\\
0&       0&                   4^n&                  2 \cdot 8^n - 2\cdot 4^n\\
0&       0&                     0&                           8^n
\end{pmatrix}. 
\end{align*}
This implies that 
\begin{align*}
\ebar_n &= \frac{2E_n}{V_n} \\
 &= 2\frac{2^n E_0 + \tfrac{3}{2} (4^n-2^n) F_0 + \left(\tfrac{11}{6}2^n - 3 \cdot 4^n + \tfrac{7}{6} 8^n \right)T_0}{V_0 + (2^n-1)E_0 + \left(\tfrac{1}{2}4^n - \tfrac{3}{2}2^n + 1 \right)F_0 + \left(\tfrac{11}{6} 2^n - 4^n + \tfrac{1}{6} 8^n - 1\right) T_0}.
\end{align*}
Considering the highest order terms with factor $8^n$ this yields that
\begin{align*}
\ebar_n \to 2 \cdot \frac{7}{6} \cdot 6 = 14, \quad \text{ as } n \to \infty. 
\end{align*}
\end{proof}

Additionally, we can estimate $\ebar$ above. 

\begin{lemma}
	\label{lem:ebar-est}
	Let $\Omega\subset\mathbb{R}^3$ be an open domain as above which is pathwise connected, and let $\chi, \chi_b$ be the topological parameters of $\overline{\Omega}$. 
Let $\mesh_0$ be a conforming simplicial triangulation of $\Omega$ with at least one interior vertex and $\mathcal{R}$ a refinement scheme satisfying~\ref{itm:R1} and~\ref{itm:R2}. 
Then, for the family of meshes $\{\mesh_n\}_{n \in \mathbb{N}}$ generated by $\mesh_{n} = \mathcal{R}(\mesh_{n-1})$ for $n \geq 1$, one has that
 \begin{align*}
 \ebar_n \leq 
  18 + \tfrac{2}{15} \max(0,7\chi+ 4 \chi_b - 96). 
 \end{align*}   
In particular, if $7 \chi + 4 \chi_b \leq  96$, this yields the upper bound $18$. 
\end{lemma}
\begin{proof}
It suffices to show that for any mesh $
 \mesh$ with mesh quantities $V, E, F, T$, for the refined mesh $\mesh' = \mathcal{R}(\mesh)$ with mesh quantities $V',E',F',T'$ one has that 
 \begin{align*}
\ebar' = \frac{2E'}{V'}
\end{align*}
satisfies the estimate. 

As in the proof of Lemma~\ref{lem:ebar-asymp} with $\mathcal{R}$ satisfying~\ref{itm:R1},~\ref{itm:R2} and applying Proposition~\ref{pro:3d-quant} we find that 
\begin{align*}
\ebar' 
&= \frac{2E'}{V'} 
= 2 \frac{2E + 3 F  + T}{V + E} \\
&= 2 \frac{2E + 6 E - 6 V - 3 V_b + 6 \chi + 3 \chi_b + E - V - V_b + \chi + \chi_b}{V + E}\\
& = 2 \frac{9E -7V - 4 V_b + 7 \chi + 4 \chi_b}{V + E}.
\end{align*}
Since $\mesh$ has at least one interior vertex, we have that $V \geq 5$, $V_b \geq 4$ and thus,
\begin{align*}
\ebar' 
&= 18  + \frac{2}{V + E} \left(-16V - 4 V_b + 7 \chi + 4 \chi_b\right)\\
& \leq
 18  + \frac{2}{V + E} \left(-96 + 7 \chi + 4 \chi_b\right).
\end{align*}
If $7\chi + 4 \chi_b \leq 96$, then the upper bound is $18$. 
Otherwise we use $V\geq 5$ and $E \geq 10$ (see Example~\ref{ex:small-mesh}) to obtain 
\begin{align*}
\ebar' 
& \leq
 18  + \tfrac{2}{15}\left(-96 + 7 \chi + 4 \chi_b\right),
\end{align*}
which proves the claim. 
\end{proof}

For a simply connected domain one has that $\chi = 1$ and $\chi_b = 2$. This implies that $7\chi + 4 \chi_b = 15 < 96$, and hence Lemma~\ref{lem:ebar-est} yields the upper bound $18$. 

\begin{example}\label{ex:refinement} 
There are various mesh refinement algorithms that satisfy the assumptions~\ref{itm:R1},~\ref{itm:R2} and consequently generate families of meshes with $\ebar$ converging to $14$. In ~\cite{PR.2003.c,PR.2005} it is shown that $\tbar$ converges to $24$ under those general conditions. 

Mesh refinement schemes satisfying the assumptions include both the generalization of the red refinement to 3D as well as bisection refinements. 
We refer to~\cite{ref:Freudenthalulation} for more properties of the refinement schemes mentioned below.  

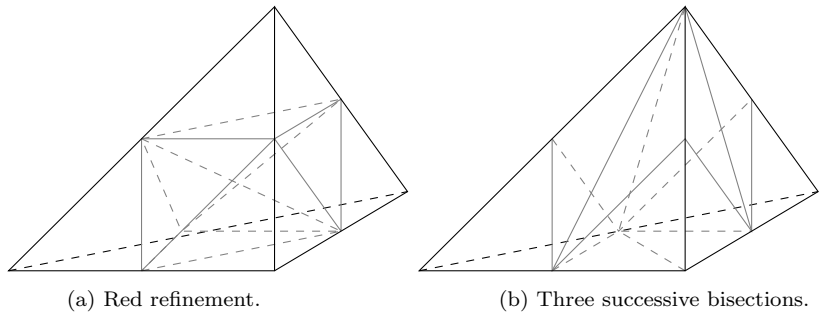
\begin{figure}[ht]
	\begin{subfigure}[b]{0.49\textwidth}
	\begin{tikzpicture}[scale=3.5]
		\draw[white] (-0.3,0) -- (-0.3,0.3);
		\draw[dashed] (0,0) -- (1.5,0.3);
		\draw[gray] (0.5,0) -- (1,0.5)--(0.5,0.5) -- (0.5,0);
		\draw[gray] (1.25,0.15)--(1.25,0.65)  -- (1,0.5) -- (1.25,0.15);
		\draw[gray,dashed](0.5,0.5) -- (1.25,0.65)  -- (0.65,0.15) -- (0.5,0.5);
		\draw[gray,dashed] (0.5,0) -- (1.25,0.15)  -- (0.65,0.15) -- (0.5,0);   
		\draw[gray,dashed] (0.5,0.5) -- (1.25,0.15);
		\draw (0,0) -- (1,0) -- (1,1) -- cycle;
		\draw (1,0) -- (1.5,0.3) -- (1,1);
	\end{tikzpicture}     
		\caption{Red refinement.}
		\label{fig:red-ref}
	\end{subfigure}
	\hfill
	\begin{subfigure}[b]{0.49\textwidth}
	\begin{tikzpicture}[scale=3.5]
		\draw[dashed] (0,0) -- (1.5,0.3);
		\draw[gray, dashed] (1,0) -- (0.75,0.15) -- (1,1);
		\draw[gray,dashed] (0.5,0) -- (0.75,0.15)--(1.25,0.15);
		\draw[gray] (0.5,0) -- (1,1)--(1.25,0.15);
		
		\draw[gray,dashed] (0.5,0.5) -- (0.75,0.15)--(1.25,0.65);
		\draw[gray] (0.5,0.5) -- (0.5,0) -- (1,0.5) -- (1.25,0.15)--(1.25,0.65);
		\draw (0,0) -- (1,0) -- (1,1) -- cycle;
		\draw (1,0) -- (1.5,0.3) -- (1,1);
	\end{tikzpicture}     
		\caption{Three successive bisections.}
		\label{fig:bisec}
	\end{subfigure}
	\caption{Refinement of a single tetrahedron.}
\end{figure}

 \begin{enumerate}[label=(\alph*)]
 \item  \label{itm:Freudenthal-ref} The Freudenthal algorithm constitutes the uniform refinement of the Freudenthal triangulation as introduced in Example~\ref{ex:freudenthal} to produce a finer Freudenthal triangulation.   
 It dates back to~\cite{F.1942} can be performed in any dimension, see Figure~\ref{fig:red-ref} for the splitting of a single tetrahedron $\tet$.
For the Freudenthal triangulation it agrees with  3D red refinement schemes independently introduced by~\cite{B.1995,Z.1995,moore1995adaptive} for general triangulations. 
	
In the 3D red refinement a tetrahedron $\tet$ is split as follows, cf. Figure~\ref{fig:red-ref}. 
Each face is split by a 2D red refinement, i.e., each edge is bisected and each face is split into 4 congruent triangles by connecting the mid points of the edges by inserting a new edge. 
This defines 4 new tetrahedra, each of which is associated to one vertex of the old tetrahedron $\tet$. 
What remains is a octahedron inside $\tet$. 
This is split into further 4 tetrahedra by inserting one of the diagonals as a new edge. 
Clearly,~\ref{itm:R1} and~\ref{itm:R2} are satisfied, and hence the above results are applicable. 
Note that there are different options of which interior edge to choose, and, e.g., choosing the longest diagonal does not preserve the shape-regularity, see~\cite{Z.1995,moore1995adaptive}. 

\item \label{itm:bisection} There are various refinements based on bisection, cf.~\cite[Sec.~3.2]{ref:Freudenthalulation}. 
Here we focus on uniform refinements, even though some of the bisection algorithms are particularly suited for adaptive mesh refinement. 

By a full uniform refinement we mean three successive bisections, each of which halves the volume. 
In this way one tetrahedron is split into $8$ subtetrahedra, each face is split into $4$ and the bisection is done in a way such that afterwards each edge is bisected once, see Figure~\ref{fig:bisec}. 
This generates new vertices exactly at the mid points of the original edges. 
Consequently,~\ref{itm:R1} and~\ref{itm:R2} are satisfied and again the above results apply. 
The following schemes differ in the way the bisection edges are chosen. 
	\begin{enumerate}[label = (\arabic*)] 
		\item 
	Let us consider the 3D generalization~\cite{B.1991,M.1994.T,LJ.1995} of the newest vertex bisection, originally introduced in 2D in~\cite{M.1991}.  
For general dimensions the refinement scheme is presented by~\cite{Maubach1995,Tr.1997} and further developed in~\cite{S.2008}. 
	It assumes a condition on the initial mesh $\mathcal{T}_0$, relaxed to the so-called matching neighbor condition in~\cite{S.2008}. 
	Under this condition uniform refinements, meaning that each simplex is refined once by the local refinement routine, are known to yield a conforming triangulation. 
	
	This refinement scheme has the advantage that it is known to preserve the shape-regularity. 
	Furthermore, it is particularly suited for adaptive mesh refinement and can be proven to satisfy optimality, see~\cite{BDD.2004}. 
	\item Choosing the longest edge as bisection edge leads to the refinement schemes introduced in 2D in~\cite{R.1984} and in 3D in~\cite{R.1991}. 
	The preservation of the shape-regularity seems to be available only in the 2D case. 
	Note that conformity of the uniform refinement is available only for special initial meshes. 
	\end{enumerate} 
		\item  \label{itm:CP}
The Carey--Plaza algorithm~\cite{PC.2000} relies one full bisection of the faces (by the 2D longest edge refinement) into 4 faces and a subsequent conforming subdivision of the tetrahedra into 8 tetrahedra. 
The authors refer to this as skeleton-based bisection. 
Since faces are split first, the uniform refinement is conforming. 
This leads to~\ref{itm:R1} and~\ref{itm:R2} being satisfied and thus the above results hold. 
This is the refinement algorithm that FEniCS (dolfin) and FEniCSx utilize. 
 \end{enumerate} 
\end{example}

Note that for the schemes we have presented, as the mesh is refined the average quantity $\ebar$ converges to the value of the regular mesh, cf. Example~\ref{ex:freudenthal}. 
However, there are also uniform refinement schemes that have other limiting values, e.g., the refinement by successive Alfeld refinement or by successive Worsey--Farin refinement. 
However, those refinement schemes do not yield shape-regular meshes.

\begin{remark}[other uniform mesh refinement schemes]\label{rmk:other-val} \hfill
	\begin{enumerate}
		\item 
For the uniform refinement resulting from successive barycentric refinement (cf.~Alfeld split) in general dimension $d \geq 2$  with the corresponding arguments as in the proof of Lemma~\ref{lem:ebar-asymp} one can show that 
	\begin{align*}
		\tbar_n \to d(d+1) \quad \text{ and } \quad \ebar_n \to 2(d+1),
	\end{align*}
asymptotically as $n \to \infty$. 
Note that this matches the limiting values for $d = 2$ in Remark~\ref{rmk:2d-av-quant}. 
In dimension $d = 3$ this yields $\tbar_n \to 12$ and $\ebar_n \to 8$, as $n \to \infty$, cf.~\cite[Thm.~5.3]{PR.2005}.  
This limiting behavior is true also for adaptive barycentric  refinement.  
However, due to the deterioration of the shape-regularity, barycentric refinement is not of practical interest. 

\item For the uniform barycentric refinement that generalizes the Powell--Sabin split to general dimensions the asymptotic mesh quantities are investigated in~\cite[Thm.~1.3]{SW.2019} and references therein. 
	For $d = 3$ the asymptotic values  $ 22$ for $\tbar$ and $ 13$ for $\ebar$ as in~\cite{PR.2005} are recovered. 
	\end{enumerate}
\end{remark}

We chose the quantity $\ebar$ to compare the dimensions of the finite element spaces, since it includes the lowest dimensional subsimplices of the meshes. 

\begin{remark}
To visualize the asymptotic behavior of the mesh quantities we consider the refinement of two initial meshes. 
The first domain $\Omega$ is the cube $[0,20]^3$ with a Freudenthal mesh $\mesh_F$ consisting of $5 \times 5 \times 5$ Kuhn cubes, see Figure~\ref{fig:init_mesh_F}.

The second domain $\Omega$ is a diamond shaped domain for $k \geq 3$. 
It is defined as the convex hull of the two spine vertices $(0,0,-20)$ and $(0,0,20)$ and $k$ equally spaced points $(x,y,0)$ on the circle with radius $20$ in the $(x,y)$-plane.  
The initial mesh consists of $2k$ tetrahedra, each of which contains the mid point $(0,0,0)$ and one of the two spine vertices. 
We refer to this mesh as $\mesh_{D,k}$, see Figure~\ref{fig:initial_mesh_diam}. 
\begin{figure}[ht]
\centering
\begin{subfigure}[b]{0.49\textwidth}
\begin{center}
    \includegraphics[width = \linewidth]{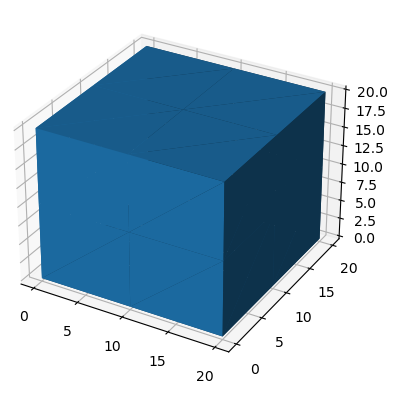}
\end{center}  \caption{Freudenthal mesh $\mesh_F$.}
\label{fig:init_mesh_F}
\end{subfigure}
     \hfill
\begin{subfigure}[b]{0.49\textwidth}
\begin{center}
  \includegraphics[width = \linewidth]{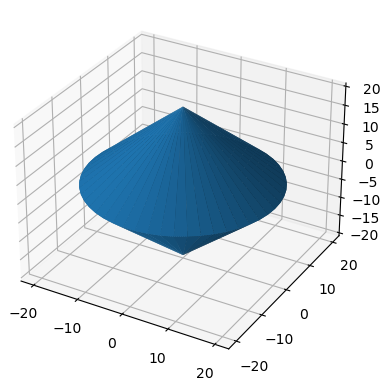}
\end{center}\caption{Diamond shaped $\mesh_{D,k}$ for $k = 50$. }
\label{fig:initial_mesh_diam}
\end{subfigure}
\caption{Initial meshes for experiment in FEniCS, generated with \texttt{tikzplotlib}. }
\label{fig:meshes}
\end{figure}

FEniCS~\cite{Fenics.2015}
 uses the Carey--Plaza algorithm~\cite{PC.2000} for mesh refinement, cf. Example~\ref{ex:refinement}~\ref{itm:CP}. 

\begin{itemize}
	\item We determine the mesh quantities for both initial meshes $\mesh_F$ and $\mesh_{D,50}$ experimentally for uniform mesh refinement. 
	For uniform mesh refinement of the regular initial mesh $\mesh_F$ the Carey--Plaza algorithm coincides with the Freudenthal algorithm in Example~\ref{ex:refinement}~\ref{itm:Freudenthal-ref} and the bisection algorithms in Example~\ref{ex:refinement}~\ref{itm:bisection}.  
	
	Figure~\ref{fig:e-unif} displays the average quantity $\ebar$ plotted over the total number of vertices contained in the meshes. 
	The experiment confirms the asymptotic behavior in the sense that $\ebar$ converges to $14$, as $V\to \infty$.  
\begin{figure}[ht]
\centering
\begin{subfigure}[b]{0.49\textwidth}
\begin{center}
 	\includegraphics[width = \linewidth]{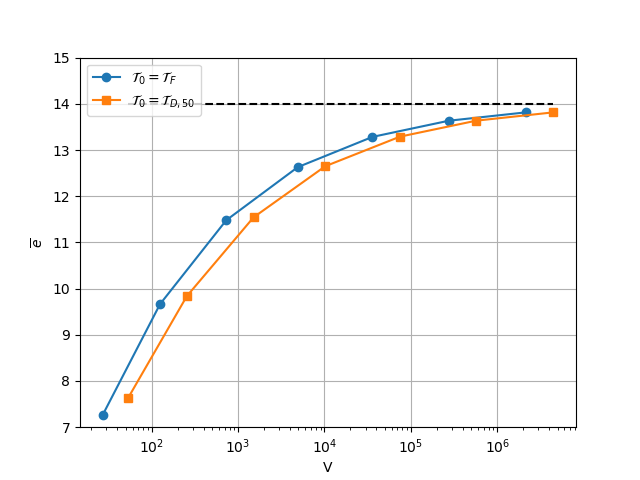}
 \end{center} \caption{Uniform mesh refinement}
\label{fig:e-unif}
\end{subfigure}
     \hfill
\begin{subfigure}[b]{0.49\textwidth}
\begin{center}
  	\includegraphics[width = \linewidth]{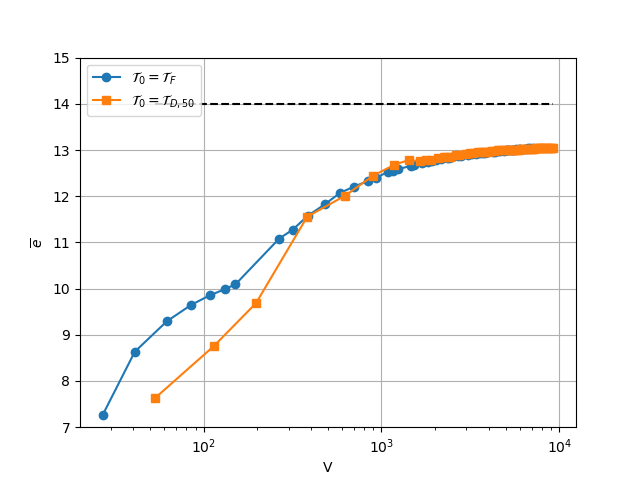}
 \end{center} \caption{Adaptive mesh refinement}
\label{fig:e-adapt}
\end{subfigure}
\caption{Behavior of $\ebar$ for a family of meshes with $V$ vertices, generated by the Carey--Plaza algorithm from initial meshes $\mesh_0$.  }
\label{fig:e-conv}
\end{figure}

\item Additionally, we consider adaptive mesh refinement with the Carey--Plaza algorithm for both initial meshes. 
In each refinement step, the simplex containing the point $v^* \in \Omega \setminus \partial \Omega$ is refined, and the conforming closure is taken. 
As refinement point we choose $v^* = (\tfrac{1}{9},\tfrac{1}{3},\tfrac{1}{3})$, which is contained in both domains.  
Figure~\ref{fig:e-adapt} shows the average quantities plotted over the total number of vertices contained in the  meshes. 
Here accuracy issues occur before the value of $\ebar$ is close to the value $14$. Indeed,  for adaptive mesh refinement the number of vertices grows much slower in the mesh refinement. Still, asymptotically the value $14$ might be approached, but the experiment does not allow us to verify this.
\end{itemize}
The FEniCS code used to produce Figures \ref{fig:meshes} and~\ref{fig:e-conv} is available in \cite{fenicscode_stokescounting}. 
\end{remark}
 
To the best of our knowledge there are no results on the asymptotic mesh quantities for adaptive mesh refinement schemes. 
This is left to future work. 

\section{Piecewise polynomials on split meshes}
\label{sec:pposm}

We now derive asymptotic formul{\ae} for the dimensions of various finite element spaces
on split 3D meshes in terms of the basic descriptors $\ebar$ and $V$ of the original mesh $\mesh$.
This will allow us to compare the dimensions of the standard Scott--Vogelius elements on the original mesh $\mesh$ with elements using lower polynomial degrees on split meshes. 

For a mesh $\widetilde{\mesh}$ and $k \geq 1$ we denote by $P_k(\widetilde{\mesh})$ the space of continuous, piecewise polynomial functions of polynomial degree at most $k$ on $\widetilde{\mesh}$. 
Then $\vhk_k(\widetilde{\mesh})$ denotes the space of vector-valued functions with each component in $P_k(\widetilde{\mesh})$, i.e., we have that $\dim(\vhk_k(\widetilde{\mesh})) = 3 \dim(P_k(\widetilde{\mesh}))$. 
Furthermore, we denote by $\Pi_{k-1}(\widetilde{\mesh})$ the space of (discontinuous) piecewise polynomials of degree at most $k-1$ on $\widetilde{\mesh}$. 

Let us first summarize the dimensions of the finite element spaces in terms of the mesh quantities $\widetilde{V}, \widetilde{E}, \widetilde{F}$, and $\widetilde{T}$, denoting the number of vertices, edges, faces, and tetrahedra in $\widetilde{\mesh}$, respectively. 
Counting the Lagrange nodes we have 
\begin{align}\begin{split}\label{eq:poly-cont}
\dim(P_1(\widetilde{\mesh})) &= \widetilde{V},\\
\dim(P_2(\widetilde{\mesh})) &= \widetilde{V} + \widetilde{E},\\
%\dim(P_3(\widetilde{\mesh})) &= \widetilde{V} + 2 \widetilde{E} + \widetilde{F},\\
%\dim(P_4(\widetilde{\mesh})) &= \widetilde{V} + 3 \widetilde{E} + 3\widetilde{F}  + \widetilde{T},\\
\dim(P_k(\widetilde{\mesh})) &= \widetilde{V} + (k-1) \widetilde{E} + \tfrac{1}{2}(k-1)(k-2)\widetilde{F} +\tfrac{1}{6} (k-1)(k-2)(k-3)\widetilde{T},
\end{split}
\end{align}
for $ k \geq 3$. 
Denoting by $P_{k-1}$ the space of polynomials of degree at most $k-1$ on a single tetrahedron, we also have 
\begin{align}\label{eq:poly-disc}
\dim(\Pi_{k-1}(\widetilde{\mesh})) = \widetilde{T} \, \dim( P_{k-1}) 
= \widetilde{T}\,  {2+k \choose k-1}
\qquad \text{ for } k \geq 1. 
\end{align} 
Then, for the pressure space $\nabla \cdot \vhk_k(\widetilde{\mesh}) \subset \Pi_{k-1}(\widetilde{\mesh})$ we obviously have 
\begin{align}
	\dim (\nabla \cdot \vhk_k(\widetilde{\mesh})) \leq \dim(\Pi_{k-1}(\widetilde{\mesh})). 
\end{align}

Now we want to compute or estimate the dimensions of the mixed finite element spaces $(\vhk_k(\widetilde{\mesh}),\nabla \cdot \vhk_k(\widetilde{\mesh}))$ on the original mesh $\mesh$ and the split meshes $\meshA$ and $\meshW$ as introduced in section~\ref{sec:splitmesh}, i.e., for $\widetilde{\mesh} \in \{\mesh, \meshA, \meshW\}$. 
For this we only need to combine the above identities~\eqref{eq:poly-cont},~\eqref{eq:poly-disc} and
 Lemma~\ref{lem:split-meshes}, which describes the mesh quantities of the split meshes in terms of the mesh quantities of the original mesh. 
Then we may apply the identities in Remark~\ref{rmk:3d-approx-quant} to express the approximate dimensions in terms of only $V$ and $\ebar$ for sufficiently large meshes. 
The resulting approximate dimensions in terms of $\ebar$ and $V$ are collected in Table~\ref{tbl:compdim}, and for the case $\ebar = 14$ they are collected in Table~\ref{tbl:compdim-14}. 

	\begin{table}[ht!] \small
  \begin{TAB}(r)[3pt]{|c|c|c|c|c|c|}{|cc|cc|c|}
    mesh $\widetilde{\mesh}$ & $\meshW$ & $\meshW$ & $\meshA$ & $\mesh$ & $\mesh$
    \\
    degree &   $k = 1$ &   $k = 2$ &  $k = 3$ & $k = 4$ & $k = 6$ 
\\   
     $\dim \vhk_k(\widetilde{\mesh})$ & $(4.5 \overline{e}-6)V$  & $(27 \overline{e} - 48)V$ & $(28.5\overline{e}-48)V$ & $(15\overline{e}-18)V$ & $(52.5\overline{e}-87)V$     \\ 
  $\dim \sdiv \vhk_{k}(\widetilde{\mesh})$
  & 
$(4 \ebar-8)V$
& 
$ (19\ebar-38)V$
  & 
  $(20\overline{e}-40)V$
  & $(10\overline{e}-20)V$ 
  & $(28\overline{e}-56)V$   
   \\ 
 total &
$(8.5\overline{e} - 14)V$
&
$(46\ebar-86)V$
 & $(48.5\overline{e}-88) V$ & $(25 \overline{e}-38)V$ & $(80.5\overline{e} - 143)V$ \\    
  \end{TAB}
  \caption{Approximate dimensions of velocity space and approximate (best available) upper bound on the dimension of the pressure space. 
  	For an original mesh $\mesh$, the Worsey--Farin split $\meshW$ and the Alfeld split $\meshA$ are defined in section~\ref{sec:pposm}. 
The quantity $\ebar$ is defined in~\eqref{eqn:ebardef} and $V$ is the number of vertices in the original mesh $\mesh$. 
The approximate dimensions for the pressure space take into account the singular edges introduced by the split but do not take into account singular edges and vertices of $\mesh$. 
The bounds on the dimensions of the pressure spaces for the Worsey--Farin split 
are justified in~\eqref{eqn:lowersypres} and \eqref{eq:dim-V2-meshW-2}. 
}
  \label{tbl:compdim}
\end{table}

	\begin{table}[htb] \small
  \centering
  \begin{TAB}(r)[3pt]{|c|c|c|c|c|c|}{|cc|cc|c|}
  mesh $\widetilde{\mesh}$  & $\meshW$ & $\meshW$ & $\meshA$ & $\mesh$ & $\mesh$  \\
    degree &  \;\;$k = 1$\;\; &  \;\; $k = 2$ \;\; &  \;\; $k = 3$ \;\; &  \;\;$k = 4$ \;\; &  \;\;$k = 6$ \;\; \\   
     $\dim \vhk_k(\widetilde{\mesh})$ & $57V$  & $330V$ & $351V$ &
$192V$ & $648V$   \\ 
   $\dim \sdiv \vhk_{k}(\widetilde{\mesh})$
   &
$48V$ &
$228V$
   & 
   $240V$ 
   &
    $120V$ 
   &
    $336V$
       \\ 
 total &
$105V$ &
$558 V$
  & $591 V$ &
$312V$ & $984V$ \\    
  \end{TAB}
  \caption{Approximate dimensions of the velocity space and approximate upper bound on the dimension of the pressure space for $\overline{e} = 14$.} 
  \label{tbl:compdim-14}  
\end{table}

One may have  $\sdiv \vhk_{k}(\widetilde{\mesh}) \subsetneq \Pi_{k-1}(\widetilde{\mesh})$, and thus the dimension of the mixed finite element space $(\vhk_k(\widetilde{\mesh}),\nabla \cdot \vhk_k(\widetilde{\mesh}))$ is slightly smaller than that of $(\vhk_k(\widetilde{\mesh}),\Pi_{k-1}(\widetilde{\mesh}))$. 
This is related to possible deficiencies in the mesh $\widetilde{\mesh}$ such as edge singularities. 

In the following we shall neglect this effect for the original mesh $\mesh$, because in general there is no strategy to count or to characterize the pressure space available. 
However, we shall take into account the singular edges generated by the splits where known, since they have a substantial effect. 

\subsection{Original mesh}\label{sec:Scott-Vog}

Recall that the standard Scott--Vogelius element~\cite{lrsBIBbk,ScottVogeliusA} $(\vhk_k(\mesh), \sdiv \vhk_k(\mesh))$ is known to be inf-sup stable for $k \geq 6$ on the regular Freudenthal triangulation as  in Example~\ref{ex:freudenthal}. 
This is proved in~\cite{zhang2011divergence}.  
Computational experiments on this mesh family~\cite{lrsBIBiu} suggest that the inf-sup condition holds for $k\geq 4$, independently of the mesh size $h$.

For this reason let us compute the approximate dimensions in both cases for general meshes for $(\vhk_k(\mesh), \Pi_{k-1}(\mesh))$, i.e., neglecting singular edges.  
By~\eqref{eq:poly-cont},~\eqref{eq:poly-disc} for $k= 6$ and $\widetilde{\mesh} = \mesh$, and applying~\eqref{eq:approx-quant} we have 
\begin{align*}
\dim \vhk_6(\mesh) &= 3 \dim P_6(\mesh) = 3(V + 5 E + 10 F + 10 T) \approx (52.5 \ebar - 87)V,\\
\dim \Pi_5(\mesh) &= 56 T \approx (28\ebar - 56) V,  
\end{align*}
compare the last column in Table~\ref{tbl:compdim}.
Using~\eqref{eq:poly-cont},~\eqref{eq:poly-disc} for $k= 4$ and $\widetilde{\mesh} = \mesh$, and applying~\eqref{eq:approx-quant} we have
\begin{align}\begin{split}\label{eq:dim-P4-T}
\dim \vhk_4(\mesh) &= 3 \dim P_4(\mesh) = 3(V + 3 E + 3 F + T) \approx (15\ebar-18)V,\\
\dim \Pi_3(\mesh) &= 20 T 
\approx (10\ebar - 20)V, % \notag
\end{split}
\end{align}
compare the fifth column in Table~\ref{tbl:compdim}.

\subsection{Alfeld split}\label{sec:Alfeld}
Alfeld splits subdivide tetrahedra into four subtetrahedra, see section~\ref{sec:splitmesh}. 
Given an initial mesh $\mesh$ we denote the resulting mesh by $\meshA$.
It is known~\cite{ref:zhang3DalfeldSplit} that the inf-sup condition is satisfied
for the pair $(\vhk_k(\meshA),\piq_{k-1}(\meshA))$ for $k\geq 3$ in $d = 3$ dimensions, see also~\cite{ref:GuzmaNeilanBarycntr,neilan2020stokes}. 
 Indeed, for this split one has that $\piq_{k-1}(\meshA)=\sdiv \vhk_k(\meshA)$. 
Let us compute the dimensions for the lowest order case $k = 3$. 

Using~\eqref{eq:poly-cont},~\eqref{eq:poly-disc} for $k= 3$ and $\widetilde{\mesh} = \meshA$, as well as Lemma~\ref{lem:split-meshes}~\ref{itm:split-A}, and~\eqref{eq:approx-quant} we find  
\begin{align} \notag
\dim \vhk_3(\meshA) &= 3 \dim P_3(\meshA) 
= 3(V_A + 2 E_A + F_A) \\
\label{eq:P2-TA}
&= 3(V + 2 E + F + 15T) 
 \approx (28.5\ebar -48)V,\\
\dim \Pi_2(\meshA) &= 10 T_A = 40T  \approx (20 \ebar - 40)V, \notag
\end{align}
compare the fourth column in Table~\ref{tbl:compdim}.

For the Scott--Vogelius elements~\cite{lrsBIBbk,ScottVogeliusA}
on $\mesh$ for $k=4$ the pressure space has only half the size, cf.~section~\ref{sec:Scott-Vog}. 
Note that a major 
contributor to the dimension is the large number of interior nodes of $\vhk_3(\meshA)$ compared to $\vhk_4(\mesh)$. 

\subsection{Worsey--Farin split}\label{sec:WF}
The Worsey--Farin split subdivides tetrahedra into 12 subtetrahedra.
Given an initial mesh $\mesh$, let $\meshW$ denote the resulting triangulation, cf. section~\ref{sec:splitmesh}. 
It is shown in~\cite{FGNZ.2022} that the inf-sup condition is satisfied for the pair $(\vhk_k(\meshW),\sdiv\vhk_k(\meshW))$ for $k\geq 1$ in $3$ dimensions. 
Therein for $k = 1$ the pressure space $\sdiv \vhk_1(\meshW)$ is characterized explicitly. 
See~\cite{ref:quadraticPowellSabinTets} for the case $k = 2$.  

Of course, $\sdiv \vhk_k(\meshW)\subset\piq_{k-1}(\meshW) $,
but it can be a strict subspace~\cite{ref:quadraticPowellSabinTets}, 
depending on the precise definition of the split.
In~\cite{FGNZ.2022}, the split is chosen so that on each
face in $\mesh$, there are 3 singular edges 
(the dashed gray lines in Figure~\ref{fig:alfeldsplit}, right).
A singular edge is an edge where exactly four faces meet, with opposite faces being coplanar. 
Thus there are 3 constraints per face~\cite[Lemma~4.3]{FGNZ.2022}. 

Let us consider the lowest-order case $k=1$. 
Using~\eqref{eq:poly-cont} for $k= 1$ and $\widetilde{\mesh} = \meshW$, Lemma~\ref{lem:split-meshes}~\ref{itm:split-WF}, and~\eqref{eq:approx-quant} we have 
\begin{align*}
	n_V \coloneqq \dim \vhk_1(\meshW) &= 3 \dim P_1(\meshW) 
	= 3 V_W = 3(V  + F + T) \\
	& \approx (4.5\ebar - 6)V,%\\
%	\dim \Pi_0(\meshW) &= T_W = 12 T  \approx (6 \ebar - 12)V, 
\end{align*}
compare the second column in Table~\ref{tbl:compdim}. 
The space $\sdiv\vhk_1(\meshW)$ is characterized in~\cite[Prop.~6.1]{FGNZ.2022} and using~\eqref{eq:approx-quant} its dimension is bounded by
\begin{align} \label{eqn:lowersypres}
	\dim\sdiv\vhk_1(\meshW) \leq
	4 (F-F_b)  + F_b 
	\approx 4F
	 \approx 4(\ebar - 2)V,
\end{align}
cf.~the first column in Table~\ref{tbl:compdim}. 
Note that this is smaller than
\begin{align*}
	n_\Pi \coloneqq	\dim \Pi_0(\meshW) &= T_W = 12 T  \approx 6( \ebar - 2)V,
\end{align*}
see~\eqref{eq:poly-disc} for $k= 1$ and $\widetilde{\mesh} = \meshW$, Lemma~\ref{lem:split-meshes}~\ref{itm:split-WF}, and~\eqref{eq:approx-quant}. 
The functions in the larger space $\piq_0(\meshW)$ but not in its subspace
$\sdiv\vhk_1(\meshW)$ are often referred to as \emph{missing modes}. 
Although there are 3 constraints per face~\cite[Lemma 4.2]{FGNZ.2022},
only two of them are linearly independent.
On the boundary, only one constraint is active. 
Thus the number of independent constraints is $2F-F_b=4T$, cf.~\eqref{eq:topocef}. 
Indeed, asymptotically for large meshes, using~\eqref{eq:topocef} the number of missing modes is approximately
\begin{align*}
	12T - 4(F-F_b) - F_b = 12T - 8T - 2F_b + 3 F_b = 4T + F_b,
\end{align*}
which is approximately $4T$ for sufficiently large meshes. 

In the introduction of~\cite{FGNZ.2022} it is suggested 
that one could work with the constraints on the full
variational space $\vhk_1(\meshW)\times \piq_0(\meshW)$ of dimension $n_V + n_\Pi$ as defined above.
This would require to use matrices of size $(n_V+n_\Pi)\times (n_V+n_\Pi)$. 
Note that the pressure space $\piq_0(\meshW)$ is larger than the velocity space $n_\Pi > n_V$ for $\ebar > 4$. 
In the Iterated Penalty Method (IPM)~\cite{lrsBIBgd,lrsBIBih}, 
the matrices are of size $n_V\times n_V$, and they are symmetric and 
positive definite, so IPM may be competitive as indicated in
\cite[Table 10]{FGNZ.2022}. 
Even if one would work directly with $\sdiv \vhk_1(\meshW)$ as pressure space,
which has dimension 
$\approx (4\ebar - 8)V$, 
this is still comparable to $n_V$. 

For the second order case, using~\eqref{eq:poly-cont},~\eqref{eq:poly-disc} for $k= 2$ and $\widetilde{\mesh} = \meshW$, as well as  Lemma~\ref{lem:split-meshes}~\ref{itm:split-WF},  and~\eqref{eq:approx-quant} we find that
\begin{equation} \label{eq:dim-V2-meshW}
\begin{split}
	\dim \vhk_2(\meshW) &= 3 \dim P_2(\meshW) 
	= 3(V_W +  E_W) = 3(V + E + 4 F + 9 T) \\ 
& \approx (27\ebar - 48)V. 
\end{split}
\end{equation}
 Thanks to a formula for the dimension $\dim\sdiv \vhk_2(\meshW)$ by \cite{ref:quadraticPowellSabinTets} and \cite[Lem.~5.10]{GLN22} we obtain
\begin{equation} \label{eq:dim-V2-meshW-2}
	\dim\sdiv \vhk_2(\meshW) = 28T+5F-1 \approx (19\ebar - 38)V,
\end{equation}
compare the second column in Table~\ref{tbl:compdim}. 

\subsection{Compare and contrast}
The lowest order case $k = 1$ on the Worsey--Farin split mesh has by far the smallest dimension 
of the methods we compare.  
But for $k=2$, the dimension is  larger than the one for the Scott--Vogelius element with $k=4$ on the original mesh,  as indicated in Tables~\ref{tbl:compdim} and~\ref{tbl:compdim-14}. 
The same is true for $k = 3$ on the Alfeld split, for which the dimensions are very close to the case of quadratic velocity functions on $\meshW$. 
Indeed, by~\eqref{eq:P2-TA},~\eqref{eq:dim-V2-meshW}, and~\eqref{eq:dim-P4-T} we obtain the asymptotic formulas $\dim \vhk_3(\meshA)= \dim \vhk_2(\meshW)+1.5\ebar V$ and
$\dim \vhk_2(\meshW)= \dim \vhk_4(\mesh)+(12\ebar-30)V$. 
Thus, for $\ebar \geq 4$ we find that, asymptotically, 
\begin{align*}
\dim \vhk_3(\meshA)\geq \dim \vhk_2(\meshW)+ 6V \geq \dim \vhk_4(\mesh)+24V.
\end{align*}

For the value $\ebar = 14$ we find that the pair of velocity and pressure spaces $(\vhk_2(\meshW),\sdiv \vhk_2(\meshW))$ and $(\vhk_3(\meshA),\Pi_2(\meshA))$ have more than double the dimension of the Scott--Vogelius element for $k = 4$, see Table~\ref{tbl:compdim-14},
i.e., the size of the velocity and pressure spaces for $k=4$ on $\mesh$ are about half that of the low-order, inf-sup stable split methods for $k\geq 2$.

Furthermore, we can see that the size of the velocity and pressure spaces 
for $k=6$ on $\mesh$ are only about a factor of 4 larger than for $k=4$ on $\mesh$. 
Also the degree $k = 6$ case is only about 50\% bigger than the split methods for $k \geq 2$. 

When using the Iterated Penalty Method only the dimension of the velocity space is relevant. 
 Table~\ref{tbl:compdim-14} allows us to draw a direct comparison also in this case, which results in the same ordering of dimensions. 

\begin{remark}[2D case]
For comparison let us derive the corresponding dimensions of inf-sup stable piecewise polynomial finite element spaces in 2D. 
For a mesh $\widetilde{\mesh}$ in 2D we have that 
\begin{align}\label{eq:poly-cont-2d-a}
\dim(\vhk_k(\widetilde{\mesh})) &= 2\Big(\widetilde{V} + (k-1) \widetilde{E} 
      + \tfrac{1}{2}(k-1)(k-2)\widetilde{T} \Big)\quad \text{ for } k \geq 1,
\end{align}
and for $k \geq 1$ that 
\begin{align}
	\label{eq:poly-disc-2d}
	\dim\big(\Pi_{k-1}(\widetilde{\mesh})\big) &= \dim( P_{k-1})  \widetilde{T}   
	=\half k(k+1) \widetilde{T} .
\end{align}

\begin{enumerate}
\item On the Alfeld split $\meshA$ the pair of spaces $(\vhk_{k}(\meshA),\Pi_{k-1}(\meshA))$ are inf-sup stable for $k \geq 2$, see~\cite{arnold1992quadratic,ref:QinThesis}, see also~\cite[Thm.~2.5]{neilan2020stokes}. 
For the case $k=2$ using Lemma~\ref{lem:meshquand-2D} and Remark~\ref{rmk:avg-2d} we thus have 
\begin{align*}
\dim(\vhk_2(\meshA)) &= 2 (V_A + E_A)
 =  2(V + E + 4 T)\\
 &= 24V - 10V_b -22 \chi + 10 \chi_b \approx 24 V,\\
 \dim(\Pi_1(\meshA)) &= 3 T_A = 9 T  = 18 V - 9 V_b - 18 \chi + 9 \chi_b \approx 18V,
\end{align*}
for meshes sufficiently large that we can neglect $V_b, \chi$, and $\chi_b$. 
\item On the Powell--Sabin split the space $(\vhk_1(\meshP),\sdiv(\vhk_1(\meshP)))$ is known to be inf-sup stable, see~\cite{ref:linearPowellSabinZhang}. 
The pressure space is contained in $\Pi_0(\meshP)$, but not fully characterized. 
Again by Lemma~\ref{lem:meshquand-2D} and Remark~\ref{rmk:avg-2d} we thus have 
\begin{align*}
	\dim(\vhk_1(\meshP)) &= 2 V_P  = 2 (V + E + T)\\
	& 
	 = 12 V - 4 V_b - 10 \chi + 4 \chi_b \approx 12 V, 
	\\
	\dim(\Pi_0(\meshP)) &=  T_P = 6 T  = 12V - 6 V_b - 12 \chi + 6 \chi_b \approx 12V,
\end{align*}
for meshes sufficiently large that we can neglect $V_b, \chi$, and $\chi_b$. 
Taking into account the singular vertices on the edges for the Powell--Sabin split we find 
\begin{align*}
	\dim(\sdiv \vhk_1(\meshP)) &\leq  6 T  - E  = 9 V - 5 V_b - 9 \chi - 5 \chi_b \approx 9 V. 
\end{align*}
\item The Scott--Vogelius element on the original mesh $\mesh$ is inf-sup stable for~$k \geq 4$, see~\cite{GS.2019}. 
For $k = 4$ using Lemma~\ref{lem:meshquand-2D} and Remark~\ref{rmk:avg-2d} we find  
\begin{align*}
	\dim(\vhk_4(\mesh)) &= 2(V + 3 E + 3 T ) = 32 V - 12 V_b - 30 \chi 
+ 12 \chi_b
\approx 32 V, 
	\\
	\dim(\Pi_3(\mesh)) &= 10 T = 20 V - 10 V_b - 20 \chi + 10 \chi_b \approx 20 V, 
\end{align*}
for meshes sufficiently large that we can neglect $V_b, \chi$, and $\chi_b$. 
\end{enumerate}
The bounds on the total approximate dimension of the pair of finite element spaces are $24V$ for the Powell--Sabin split and $k = 1$, $42 V $ for the Alfeld split and $k = 2$, and $52V$ for the original mesh and $k = 4$ (lowest order Scott--Vogelius element). 
The lowest order element on the Powell--Sabin has again the smallest dimension. 
Then, on the Alfeld split mesh the element of order $k = 2$ follows with less than double the dimension. 
And finally, the Scott--Vogelius element with $k = 4$ on the original mesh has only a slightly higher dimension. 
\end{remark}

\subsection{Alternative mixed methods}

There have been remarkable advances regarding exactly divergence-free methods in 3D, of which we have investigated some polynomial ones above. 
However, one should contrast this with alternatives such as the work-horse of fluid simulation, the lowest-order
(quadratic velocity, linear pressure) Taylor--Hood method~\cite{case2011connection}. 

Indeed, we may compare the exactly divergence-free finite elements, 
which have discontinuous, piecewise polynomial pressure, with
\begin{itemize}
	\item low-order 
inf-sup stable but 
 not exactly divergence-free methods, and with 
	\item exactly divergence-free methods with more complex velocity space arising as the $\curl$ of some $C^1$ conforming space. 
\end{itemize}

\subsubsection{Approximately divergence-free methods}\label{subs:approx-div}
For the Freudenthal mesh a comparison of the dimensions of low-order elements in general dimension is presented in~\cite[Sec.~3]{DST.2022}. 
In 3D for the choice $\ebar = 14$ this comes down to total dimension (of pressure and velocity space) 
\begin{itemize}
\item $21V$ for the Bernardi--Raugel  element~\cite{BR.1985},
\item $22V$ for the MINI element~\cite{ABF.1984},
\item $25V$ for the lowest-order Taylor--Hood element~\cite{TH.1973},
\end{itemize}
see~\cite[Table~2]{DST.2022}. 

Note 
however, 
that in 3D the smallest element available is not 
any of these elements, but rather
the \emph{reduced Taylor--Hood element}, first presented in~\cite[Sec.~3.1]{DST.2022}. 
Its velocity space consists of continuous piecewise affine functions and tangential edge bubble functions, and the pressure space consists of continuous piecewise affine functions. 
This results in a total dimension of
\begin{itemize}
	\item $11V$ for the reduced Taylor--Hood element \cite{DST.2022}. 
\end{itemize}
In terms of approximation order the Bernardi--Raugel, MINI, and reduced Taylor--Hood element have velocity spaces that are enriched $P_1$ elements, whereas the Taylor--Hood has $P_2$ velocity. 

For higher-order approximately divergence-free methods, note that the $4$th order Taylor--Hood velocity space is the same as the Scott--Vogelius velocity space,
but the Taylor--Hood pressure space is larger than the Scott--Vogelius pressure space.
On the other hand, one need not work directly with the pressure space for Scott--Vogelius element. 
Instead, one may utilize the Iterated Penalty Method to solve directly for the
divergence-free velocity~\cite{lrsBIBih}, with the additional benefit that the 
linear systems are positive definite.

\subsubsection{Exactly divergence-free with other velocity spaces}

In~\cite{GN.2014b} mixed finite elements with exact divergence constraints are derived from the 3D generalization of the $C^1$ conforming Zienkiewicz elements using rational bubble functions.  
The reduced version of this has total dimension of velocity and pressure space $$3V + F  + T \approx (3 + 12 + 6)V = 21V.$$ 
Note that this is the same as for the Bernardi--Raugel element and indeed the same degrees of freedom are used. 

This comparison shows that there is a price to pay for exact divergence constraints in combination with a Lagrange velocity space. 
If one of the requirements is dropped, then elements with considerably smaller dimension can be chosen. 

\section{Insights into the discrete Stokes complex}
\label{sec:compldig}

Counting the degrees of freedom can provide conceptual advances and we give one example of an application here. 
More specifically, we use this to obtain a result on a discrete Stokes complex.  
Although this is a small advance we can envisage further possible applications by 
the relevance of such techniques in the historical context. 

The simple act of counting the degrees of freedom in a two-dimensional finite element 
space~\cite{strang1973piecewise} led to a flood of research involving diverse areas
of mathematics, including algebraic geometry~\cite{billera1988homology}. 
One result~\cite{lrsBIBaf} became a critical part of the Stokes complex 
\cite{farrell2021reynolds} in two dimensions.
The counting techniques introduced here are very primitive by comparison, but they provide
a way to count in three dimensions that may be useful in other contexts.

One application of the above counting strategy is to address the question of 
identifying the precursor space in a discrete Stokes complex. 
By \emph{precursor space} of a velocity space $\vhk_k(\mesh) \subset C(\overline{\Omega})$, we mean a subspace of  $\vhk_{k+1}(\mesh)$, the curl of which coincides with the space of divergence-free functions in $\vhk_k(\mesh)$, 
\begin{align}\label{eqn:zeetudef}
	\bZ_k(\mesh) \coloneqq \set{\vv\in \vhk_k(\mesh)}{\sdiv\vv=0}.
\end{align}
Thus, such a precursor space would be part of a discrete Stokes complex together with $\vhk_k(\mesh)$ and the pressure space $\sdiv \vhk_k(\mesh)$. 

In $d = 2$ dimensions the identification of the full Stokes complex is relatively
simple~\cite{farrell2021reynolds}. 
The space $\bZ_k(\mesh)$ of divergence-free continuous piecewise polynomial functions can be represented by the $\curl$ of all $C^1$ piecewise polynomials of degree $k+1$. 
This means, that the precursor space is the space of scalar functions $S_{k+1}(\mesh) \coloneqq P_{k+1}(\mesh) \cap C^1(\overline{\Omega})$, satisfying $\curl(S_{k+1}(\mesh)) = \bZ_k(\mesh)$. 
Often it is possible to identify the space of all $C^1$ piecewise polynomials. 
Indeed, for degree $k \geq 3$ a nodal basis is available
\cite{lrsBIBaf,lrsBIBim}. 
For some split mesh methods the same is true for lower polynomial degree. 
For example, on the Malkus split the Powell basis~\cite{powell1973piecewise} of $C^1$ piecewise quadratics is well known. 

In $d = 3$ dimensions the precursor space of $\vhk_k(\mesh)$ consists of vector-valued functions and shall be denoted by $\bUps_{k+1}(\mesh) \subset \vhk_{k+1}(\mesh)$. 
By the assumption that $\vhk_k(\mesh)$ is conforming, the curl of functions in $\bUps_{k+1}(\mesh)$ has to be continuous. 
But this does not necessarily require that
$\bUps_{k+1}(\mesh) \subset C^1(\overline{\Omega})$. 
Nevertheless, it is interesting to investigate spaces of $C^1$ piecewise 
polynomials on various split meshes, as split meshes were originally developed to construct such spaces. 
At the very least, their $\curl$ is a subspace of the divergence-free space. 
For this purpose we denote 
\begin{align}\label{def:C1-precursor}
	\bS_r(\mesh) \coloneqq \vhk_{r}(\mesh) \cap C^1(\overline{\Omega}),\quad \text{ for }  r \in \mathbb{N}. 
\end{align}
Since we have that $\curl \bS_{k+1}(\mesh) \subset \bZ_k(\mesh) =  \{\vv \in \vhk_k(\mesh) \colon \sdiv \vv = 0\}$, 
it follows that $\bS_{k+1}(\mesh) \subset \bUps_{k+1}(\mesh)$. 

For the generalized Powell--Sabin split as mentioned in section~\ref{sec:splithreed} above, each tetrahedron is split into $24$ tetrahedra. 
With $\meshP$ denoting the resulting mesh, the space $\bS_2(\meshP)$ is studied in~\cite{ref:worseypipernotfarin}. 
With the arguments above it is a candidate for a precursor space of $\vhk_1(\meshP)$. 
To the best of our knowledge nothing is known about the inf-sup stability of $(\vhk_{1}(\meshP), \sdiv \vhk_{1}(\meshP))$ in three dimensions. 

On the Worsey--Farin split $\meshW$ as introduced in section~\ref{sec:splithreed} it is also not clear how to find a precursor space of $\vhk_{1}(\meshW)$. 
If we increase the degree to $k=2$, then the original Worsey--Farin
paper~\cite{ref:cubicWorseyFarin} provides the space $S_3(\meshW)$, the scalar analog of $\bS_3(\meshW)$. 
The vectorial version $\bS_3(\meshW)$ serves as candidate for a precursor space of $\vhk_2(\meshW)$. 
For each component, the degrees of freedom of $\bS_3(\meshW)$ are 4 at each vertex 
and 2 on each edge, so with~\eqref{eq:approx-quant} we find
\begin{equation}\label{eqn:essthred}
	\dim \bS_3(\meshW)=3(4V+2E)\approx 3(\ebar+4)V,
\end{equation}
for dimension $d = 3$. 
Recalling the notation %
%\begin{equation}\label{eqn:zeetudef}
$\bZ_2(\meshW) = \set{v\in \vhk_2(\meshW)}{\sdiv  v=0}$,  
%\end{equation}
%
%Then, 
the gap between 
\begin{align*}
	\curl \bS_3(\meshW) \subset \bZ_2(\meshW),
\end{align*}
can be significant, as the following theorem shows. 

\begin{theorem}\label{thm:precurwf}
For a mesh $\mesh$ in 3D sufficiently large that we may neglect the boundary vertices and $\chi_b, \chi$, cf.~Remark~\ref{rmk:3d-approx-quant}, we consider its Worsey--Farin split $\meshW$ (cf.~section~\ref{sec:splithreed}). 
	For $\ebar>4.4$,
 the space $\bZ_2(\meshW)$ defined in~\eqref{eqn:zeetudef} 
is strictly larger than the space $\curl \bS_3(\meshW)$, with $\bS_3(\meshW)$ as defined in~\eqref{def:C1-precursor}. 
	In particular, this is the case for large regular meshes such as the Freudenthal triangulation, cf.~Example~\ref{ex:freudenthal}, and for large meshes that arise from sufficiently many uniform mesh refinements, cf.~Lemma~\ref{lem:ebar-asymp}. 
\end{theorem}
\begin{proof}
	Applying~\eqref{eq:dim-V2-meshW} and~\eqref{eq:dim-V2-meshW-2} we obtain with~\eqref{eq:approx-quant} that
	\begin{equation}\label{eqn:zeetuvar}
		\begin{split}
			\dim \bZ_2(\meshW)&= \dim \vhk_2(\meshW)-\dim\sdiv\vhk_2(\meshW) \\
			&\geq 3(V+E+4F+9T)-(28T +5F) \\
			& \approx (27\ebar - 48)V - (19\ebar - 38)V =  (8\ebar-10)V.
		\end{split}
	\end{equation}
	From~\eqref{eqn:essthred} it follows that
	\begin{align*}
		\dim\curl \bS_3(\meshW) \leq
		\dim \bS_3(\meshW) =3(4V+2E)  \approx 3(\ebar+4)V.
	\end{align*}
	Combining both estimates we find that 
	\begin{equation}\label{eqn:notzeetuvar}
		\begin{split}
			\dim \bZ_2(\meshW)- \dim\curl \bS_3(\meshW)
		&		
		\geq
		 3V+3E+7F-T - 3(4V+2E)\\
			&
			=-9V-3E+7F-T\\
			&
			\approx  \big(5 \ebar-22\big)V.
		\end{split}
	\end{equation}
\end{proof}

For this reason the precursor space of $\vhk_2(\meshW)$ has to be larger than $\bS_3(\meshW)$. 

\section{Conclusions and perspectives}
\label{sec:conclu}

We use a way of counting mesh quantities in three dimensions in terms of the average number of edges $\ebar$ meeting at a vertex.
For $\ebar$ we presented upper and lower bounds. 
Furthermore, we reviewed asymptotic limits for a range of uniform mesh refinement schemes. 
These allowed us to compare the dimensions of various finite element spaces.  
We applied this to pairs of finite element spaces for problems with 
a divergence constraint, such as the incompressible Navier--Stokes equations. 
The mesh-counting techniques may be of independent interest in other contexts.

 Although the use of split meshes lowers the degree for which finite element pairs with exact divergence constraints are stable, it may not yield smaller spaces. 
Indeed, only the lowest order example on the Worsey--Farin split has smaller dimension than the Scott--Vogelius element of polynomial degree $4$ on the original mesh. 
The fact that there is only one interior node per tetrahedron limits the size. 
The largest contributors are the face nodes, since they are nearly twice as plentiful as edge nodes. 
This means that when only considering the dimensions of the spaces, the most attractive higher order method is the Scott--Vogelius element for polynomial degree $k = 4$. 

When giving up either the exact divergence constraint or the fact that the velocity space is a Lagrange space, then elements with dimensions smaller by a factor of at least $5$ are available, see subsection~\ref{subs:approx-div}.

Perhaps the main advantage of split meshes is that issues related to nearly singular vertices and edges are avoided. 
Still, the lower dimension of standard Scott--Vogelius methods of higher polynomial degree motivates to understand how to ameliorate nearly singular simplices in three dimensions.

\section{Acknowledgments}

We thank Patrick Farrell and Michael Neilan for valuable discussions and suggestions. 
We are grateful to the anonymous referees for their careful reading and helpful suggestions. 

%% New biber style  (Remove \i!!!)
\printbibliography 

%\bibliographystyle{abbrv}
%\bibliographystyle{alpha}
%\bibliography{tdstorefs}

\end{document}